\documentclass[]{article}
\usepackage{authblk}
\usepackage[letterpaper, left=3cm, right=3cm, height=22cm]{geometry}
\usepackage[english]{babel}
\usepackage[utf8]{inputenc}
\usepackage[charter]{mathdesign}
\usepackage[mathscr]{euscript}
\usepackage{amsmath,amsthm}
\usepackage{tikz-cd}
\usepackage[all]{xy}
\usepackage{todonotes}
\usepackage{hyperref}
\usepackage{bussproofs}
\newtheorem{teorema}{Theorem}[section]
\newtheorem{cor}[teorema]{Corollary}
\newtheorem{lem}[teorema]{Lemma}
\newtheorem{prop}[teorema]{Proposition}

\theoremstyle{definition}
\newtheorem{defin}[teorema]{Definition}
\newtheorem{re}[teorema]{Remark}

\newtheorem{Exs}[teorema]{Examples}

\newcommand{\kmodto}{\mathrel{\mathmakebox[\widthof{$\xrightarrow{\rule{1.45ex}{0ex}}$}]{\xrightharpoonup{\rule{1.45ex}{0ex}}\hspace*{-2.8ex}{\circ}\hspace*{1ex}}}}

\newcommand{\Multi}[1]{(L,{#1})\mbox{-} \mathtt{Cat}}

\newcommand{\Vcat}{V \mbox{-} \mathtt{Cat}}
\newcommand{\Vcats}{V \mbox{-} \mathtt{Cat}_{\mathtt{sep}}}
\newcommand{\Sets}{\mathtt{Set}}
\newcommand{\Dist}[1]{{#1} \mbox{-}\mathtt{Dist}}
\newcommand{\Mat}[1]{ {#1} \mbox{-}\mathtt{Rel}}
\newcommand{\Mon}[2]{\mathtt{Mon}({#1}, \boxtimes_{#2}, V)}
\newcommand{\Alg}[1]{\mathtt{Set}^{\mathtt{P}_{#1}}}

\newcommand{\mult}{\ast}
\newcommand{\Yo}[1]{\mathbf{y}_{#1}}
\newcommand{\CoYo}[1]{\mathbf{\lambda}_{#1}}
\newcommand{\op}{\mathtt{op}}
\newcommand{\Pow}[1]{ \mathtt{P}_{{#1} }}
\newcommand{\Coco}[1]{ \mathtt{CoCts}({#1})}
\newcommand{\Sup}{ \mathtt{Sup}}
\newcommand{\Tensor}[1]{ \boxtimes_{#1}}
\newcommand{\Mod}{V \mbox{-}\mathtt{Mod}}
\newcommand{\Inj}[1]{\mathtt{CoCts}((L,{#1})\mbox{-}\mathtt{Cat}_{\mathtt{sep}})}
\newcommand{\Quant}[1]{(V \downarrow \mathtt{#1})_{\spadesuit}}
\newcommand{\Bim}{\mathtt{Bim}}
\newcommand{\dist}{\circ}
\newcommand{\ldist}{\bullet}
\DeclareMathOperator{\Id}{Id}
\usepackage{mathtools}

\makeatletter
\def\slashedarrowfill@#1#2#3#4#5{%
	$\m@th\thickmuskip0mu\medmuskip\thickmuskip\thinmuskip\thickmuskip
	\relax#5#1\mkern-7mu%
	\cleaders\hbox{$#5\mkern-2mu#2\mkern-2mu$}\hfill
	\mathclap{#3}\mathclap{#2}%
	\cleaders\hbox{$#5\mkern-2mu#2\mkern-2mu$}\hfill
	\mkern-7mu#4$%
}
\def\rightslashedarrowfill@{%
	\slashedarrowfill@\relbar\relbar\mapstochar\rightarrow}
\newcommand\xslashedrightarrow[2][]{%
	\ext@arrow 0055{\rightslashedarrowfill@}{#1}{#2}}
\makeatother

\newcommand\restr[2]{{
		\left.\kern-\nulldelimiterspace 
		#1 
		\vphantom{\big|} 
		\right|_{#2} 
}}

\newcommand{\scomment}[1]{\ \ifhmode\todo{#1}\else\vadjust{\todo{#1}}\fi}

\tikzset{%
	symbol/.style={%
		draw=none,
		every to/.append style={%
			edge node={node [sloped, allow upside down, auto=false]{$#1$}}}
	}
}

\title{Quantale-Enriched Multicategories Via Actions}
\author{Eros Martinelli}
\providecommand{\keywords}[1]{\textbf{\textit{Keywords:}} #1}
\affil{Center for Research and Development in Mathematics and
Applications,\\
Department of Mathematics,\\
University of Aveiro,\\
Portugal.\\
\textit{E-mail adress}: \texttt{eros.martinelli@ua.pt}}
\date{}

\begin{document}

\maketitle
\begin{abstract}
  In this communication, motivated by a classical result that relates cocomplete quantale-enriched categories to modules over a quantale, we prove a similar result for quantale-enriched multicategories.

\end{abstract}
\medskip

\section{Introduction}
Lawvere, in his seminal paper \cite{MSCC}, made the important observation that fundamental mathematical structures do not only constitute the objects of a category but are themselves categories. In fact, it has been known for a long time that ordered sets can be seen as categories enriched in the two element boolean algebra; moreover, monotone maps between them are exactly enriched functors. As a leading example, Lawvere explains how metric spaces fit into his thesis by showing how they are instances of enriched categories and how results from enriched category theory are able to capture important metric constructions.\par \medskip
We must point out that, although it is possible to develop enriched category theory in the more general setting in which the enrichment is taken in a closed symmetric monoidal category, in many cases it is sufficient to take the enrichment in a commutative quantale $V$, that is to say a monoid in the monoidal category of suplattices. This leads to the notion of quantale-enriched categories which can be seen as a generalization of the notion of ordered sets where one substitutes the ordered relation with a more general relation---called enriched structure---with values in the quantale $V$.\par \medskip
Since quantale-enriched categories are a generalization of ordered sets, it is natural to ask which relations there are between the two. The very first observation is that to every quantale-enriched category $X$ we can associate an ordered set, called the underlying ordered set of $X$; its order relation relates elements of $X$ whose value under the enriched structure of $X$ is greater or equal than the unit of $V$. This construction is part of a right adjoint functor between the category of quantale-enriched categories and the category of ordered sets. Due to the form this functor has, any hope to recovery the structure of a quantale-enriched category $(X,a)$ from its underlying ordered set is going to be disappointed. In order to mantain such hope, we we must add some structure to the category of ordered sets; a structure that must contain the information which gets lost in the discretization procedure: the values of the enriched relation at elements of $X$.\par \medskip
The solution to this problem is to consider ordered sets equipped with a suitable action of the base quantale subject to conditions that allow us to define a copowered enriched category, where the copower becomes the action itself. This association will give us an equivalence between the category of ordered sets equipped with such an action and the category of copowered categories (see \cite{GORDON1997167} for the general construction).
The aforementioned equivalence restricts to an equivalence between the category of cocomplete quantale-enriched categories and the category of cocomplete ordered sets equipped with an action of the base quantale, also called the category of modules (see \cite{Pedicchio1989}).
These last two equivalences allow us to reason about enriched categories by using order theoretic arguments. An example where this is not only useful, but it has proven to be essential, is given by the results contained in \cite{HN18}, where, in order to obtain the duality between metric compact Hausdorff spaces and (suitably defined) finitely cocomplete categories enriched in the unit interval $[0,1]$, the representation of the latter as ordered sets with an action of $[0,1]$ is essential.\par \medskip
In \cite{DGAMD} D. Hofmann and G. Gutierres proved that a similar result holds also for approach spaces. Approach spaces are particular examples of $(T,V)$-categories (see \cite{hofmann2014monoidal}) where the monad $T$ is specialized to the ultrafilter monad $U$ and $V$ is specialized to the quantale $[0, \infty]^{\op}$. For these categories, Clementino, Hofmann and Tholen showed how it is possible to develop many constructions that come from enriched category theory in the more general context of $(T,V)$-categories (see \cite{HOFMANN2011283, CH09a, Hof14}). In particular, in \cite{HOFMANN2011283}, Hofmann showed how algebras for a Kock-Zöberlein monad, which generalizes the presheaf monad, characterize cocomplete $(T,V)$-categories. By using the machinery of $(T,V)$-categories, D. Hofmann and G. Gutierres proved that separated (i.e. $T_0$) cocomplete approach spaces are equivalent to continuous lattices (cocomplete topological spaces in the $(U,2)$ setting) equipped with an action of the quantale $[0, \infty]^{\op}$. \par \medskip
The aim of this paper is to prove that a similar result holds also for quantale-enriched multicategories. We also notice that quantale-enriched multicategories, from now on called $(L,V)$-categories, are particular examples of $(T,V)$-categories where the monad $T$ is specialized to the list monad $L$.
We prove that the category $\Inj{V}$ of cocomplete separated $(L,V)$-categories is equivalent to the category of quantales (ordered cocomplete multicategories) equipped with a suitable action of $V$ and denoted $\Mod(\mathtt{Quant})$.\par \medskip
We must point out that, although approach spaces and quantale-enriched multicategories are both examples of $(T,V)$-categories, the strategy used to prove the main result of this paper bears little relationship to the one used in \cite{DGAMD}; while the latter relies on a careful study of weighted $(U,[0, \infty]^{\op})$-colimits, the former essentially relies on the fact that we can internalize the notion of monoid in every monoidal category. The deep reason why approach spaces and quantale-enriched multicategories behave differently is an interesting open question the author wants to investigate; the hope is to provide a more general theory of ``actions'' for $(T,V)$-categories.\par \medskip

The structure of the paper is as follows:
\begin{itemize}
	\item In the first section we introduce some background material on $V$-categories. We briefly sketch the equivalence between the category of cocomplete quantale-enriched categories and the category of modules: \[\Coco{\Vcats} \simeq \Mod.\]
	\item In the second section we introduce $(L,V)$-categories. We show how many constructions that come from enriched category theory can be developed in the more general context of $(L,V)$-categories.
	\item The third section contains the first step towards our desired result. We analyze further the equivalence $ \Coco{\Vcats} \simeq \Mod$. First we prove that both categories can be equipped with a monoidal structure, then we prove that the aforementioned equivalence extends to the corresponding categories of monoids: \[\Mon{\Mod}{V} \simeq \Mon{\Coco{\Vcats}}{V}.\]
	\item In the fourth section we study further the category $\Mon{\Mod}{V}$. We prove that ${\Mon{\Mod}{V}}$ is equivalent to a particular subcategory of $V \downarrow \mathtt{Quant}$.
	\item In the fifth section we study further the category $\Mon{\Coco{\Vcats}}{V}$. We prove that it is monadic over $\Sets$ and that it is equivalent to $\Inj{V}$, the category of cocomplete separated $(L,V)$-categories.
	\item In the last section we collect everything together and prove that $\Inj{V}$ is equivalent $\Mod(\mathtt{Quant})$, the category of quantales equipped with a suitable action of $V$.
	\item In the appendix we recall some useful materials from \cite{JACOBS199473} about strong commutative monads we use across the paper.
\end{itemize}
\section{Preliminaries on Quantale-Enriched Categories}
In this section we recall/introduce some basic notions of $V$-categories. Our point of view is slightly different from the more "standard" one contained in \cite{ECT}, it is more "relational": following \cite{BETTI1983109, CLEMENTINO200313}, we introduce the \textit{quantaloid} of $V$-relations and we define $V$-categories starting from there. This might be seen as an overkill, but it will be clear in the section related to $(L,V)$-categories how this approach allows us to smoothly introduce some concepts also in the $(L,V)$-case.
\subsection{V-Categories and V-Functors}
\begin{defin}
	A quantale $(V, \otimes, k)$ is a complete lattice endowed with a multiplication $\otimes : V \times V \rightarrow V$ that preserves suprema in each variable and for which $k \in V$ is the neutral element. If $k \neq \perp,$ we call $V$ non-trivial.
\end{defin}
\begin{re}
	When we talk about quantale-enriched categories we always assume our base quantale $V$ to be commutative.
\end{re}
\begin{re}
	In this paper we assume---unless explicitly stated---our quantales to be non-trivial.
\end{re}
\begin{re}
	By the adjoint functor theorem applied to ordered sets, it follows that $- \otimes =$ admits a right adjoint (in each variable) denoted by $[-,=]$ and called "internal hom".
\end{re}
\begin{defin}
	Let $(V, \otimes_V, k_V)$ and $(Q, \otimes_Q, k_Q)$ be quantales. A \textit{morphism of quantales} is a suprema preserving map $f : V \rightarrow Q$ such that, for all $v,w \in V$,
	$$ f (v) \otimes_Q f(w) = f(v \otimes_V w), \ \ k_Q  = f(k_V).$$
\end{defin}
\begin{Exs}
	\begin{enumerate}
		\item The two-element boolean algebra $\mathbf{2} = \{0,1\}$ with $\wedge$ as multiplication and $\Rightarrow$ as internal hom is a quantale.
		\item More generally, every frame becomes a quantale with the multiplication given by $\wedge$. In this case we have $k = \top,$ where $\top$ is the top element of the frame.
		\item $[0, \infty]^{\op}$ (with the opposite of the natural order) with $+$ as multiplication is a quantale. The internal hom is given by "truncated minus" defined as $ [u,v] =v\ominus u = \mathtt{max}(v-u,0)$.
		\item Consider the set
		$$\Delta = \{\psi :[0, \infty] \rightarrow [0,1] \mbox{ $\mid$ } \mbox{for all } \alpha \in [0, \infty] \mbox{ : } \psi(\alpha) = \bigvee_{\beta < \alpha } \psi(\beta)\}$$
		of distribution functions. With the pointwise order it becomes a complete ordered set. For all ${\psi, \phi \in \Delta}$ and $\alpha \in [0, \infty]$, define:
		$$ \psi \otimes \phi (\alpha) = \bigvee_{\beta + \gamma < \alpha } \psi(\beta )\mult \phi(\gamma),$$
		where $\mult$ is the ordinary multiplication on $[0,1]$.
		It is shown in \cite{De} that $(\Delta, \otimes, k)$ is a quantale, where $k(0) =0$ and, for all $\alpha > 0$, $k(\alpha) = 1$.
\end{enumerate}
\end{Exs}
As we stated in the introduction of this section, we are going to present $V$-categories from a more "relational" point of view. The first step is to define the \textit{quantaloid of $V$-relations} which is the enriched generalization of the category $\mathtt{Rel}$ of (ordinary) binary relations. For an account on quantaloids we refer to \cite{STUBBE201495} for a brief overview and to \cite{DF} for a more in depth description.\par\medskip
The quantaloid $\Mat{V}$ is the order-enriched category whose objects are sets, and an arrow $r : X \xslashedrightarrow{} Y$ is given by a function
$$ r : X \times Y \rightarrow V.$$
The composition of $r : X \xslashedrightarrow{} Y $, $s : Y \xslashedrightarrow{} Z$ is given by "matrix multiplication" and is defined pointwise as
$$ s \dist r (x,z) = \bigvee_{y \in Y} r(x,y) \otimes s(y,z).$$
The identity arrow $ \Id : X \xslashedrightarrow{} X$ is
$$
\Id(x_1,x_2) =
\begin{cases}
k &\mbox{ if } x_1 = x_2, \\
\perp &\mbox{ if } x_1 \neq x_2.
\end{cases}
$$
The complete order on $\Mat{V}(X,Y)$ is the one induced (pointwise) by $V$, i.e.
\begin{equation}
\label{Ord}
 r \leq r' \mbox{ in } \Mat{V}(X,Y) \mbox{ whenever } r(x,y) \leq r'(x,y) \mbox{ in } V \mbox{ for all } x,y \in X,Y.
\end{equation}
\begin{re}
	Notice that $\Mat{V}(X,Y)$ is complete because $V$ is so. Since the multiplication of $V$ preserves suprema in both variables and because suprema commute with suprema, one has
	$$(\bigvee_i r_i) \dist (\bigvee_j s_j) = \bigvee_{i,j} r_i \dist s_j.$$
	This proves that $\Mat{V}$ is a quantaloid.
 \end{re}
\begin{re}
	Notice that for $V = \mathbf{2}$, $\Mat{\mathbf{2}}$ is the quantaloid of relations, and the "matrix multiplication" defined previously becomes the "classical" relational composition.
\end{re}
\begin{re}
	Note that \eqref{Ord} is equivalent to
	$$ k \leq \bigwedge_{x,y \in X,Y} [r(x,y), r'(x,y)].$$
	Indeed, consider $x,y \in X,Y$, then we have
	$$ r(x,y) \leq r'(x,y) \iff  k \otimes r(x,y) \leq r'(x,y) \iff  k \leq [r(x,y), r'(x,y)].$$
\end{re}
\begin{re}
	Notice that every function $f : X \rightarrow Y$ can be seen as a $V$-relation as follows:
	$$
	f(x,y) =
	\begin{cases}
	k &\mbox{ if } f(x) =y, \\
	\perp &\mbox{ if } f(x)\neq y.
	\end{cases}
	$$
	The identity in $\Mat{V}(X,X)$ is an example of this construction.
\end{re}

We have also an involution $(-)^{\circ} : \Mat{V}^{\op} \rightarrow \Mat{V}$ defined as $r^{\circ}(y,x) = r(x,y)$, which satisfies
$$ (1_X)^{\circ} = 1_X, \ \ (s \dist r )^{\circ} = r^{\circ} \dist s^{\circ}, \ \ (r^{\circ})^{\circ} = r.$$
\begin{defin}
	A $V$-category is a pair $(X,a)$, where $X$ is a set and $ a : X \xslashedrightarrow{} X$ is a $V$-relation that satisfies
	\begin{itemize}
		\item $\Id \leq a;$
		\item  $a \dist a \leq a$.
	\end{itemize}
\end{defin}
\begin{re}
	In this paper, when the $V$-structure is clear from the context, we will denote a $V$-category $(X,a)$ simply as $X$.
\end{re}
\begin{defin}
	Let $(X,a)$ and $(Y,b)$ be $V$-categories. A $V$-\textit{functor} $f : (X,a) \rightarrow (Y,b)$ is a function between the underlying sets such that
	$$a \leq f^{\circ} \dist b \dist f,$$
	which, in pointwise terms, means that, for all $x,y \in X,$
	$$ a(x,y) \leq b(f(x), f(y)).$$
	If the equality holds, we call $f$ \textit{fully faithful}.
\end{defin}
\begin{Exs}\label{ExamplesVcat}
\begin{enumerate}
		\item For $V = \mathbf{2}$, a $\mathbf{2}$-category is an ordered set and a $\mathbf{2}$-functor is a monotone map. Notice that the order relation of a $\mathbf{2}$-category $(X, \leq_X)$ does not need to be antisymmetric.
		\item Categories enriched in the quantale $[0, \infty]^{\op}$, as first recognized by Lawvere in \cite{MSCC}, are generalized metric spaces and $[0, \infty]^{\op}$-functors between them are non-expansive maps.
		\item Categories enriched in $\Delta$ are probabilistic metric spaces, as first recognized in \cite{De}.
		\item The quantale $V$ defines a $V$-category with the $V$-structure given by its internal hom $[-, =]$.
		\item By using the involution $(-)^{\circ} : \Mat{V}^{\op} \rightarrow \Mat{V}$, for every $V$-category $(X,a)$, one can define its opposite category $X^{\op} = (X, a^{\circ}).$
		\item Let $(X,a)$ and $(Y,b)$ be $V$-categories. We define the $V$-category formed by all $V$-functors ${f : (X,a) \rightarrow (Y,b)}$, denoted by $([X,Y],[X,Y](-,=) )$, with the following $V$-structure:
		$$[X,Y](f,g) = \bigwedge_{x\in X} b(f(x),g(x)).$$
		In particular we have two very important $V$-categories:
		$$\mathbb{D}(X)= [X^{\op},V], \mbox{ the category of presheaves},$$
		$$\mathbb{U}(X) = [X,V]^{\op}, \mbox{ the category of co-presheaves.}$$
		Notice that they are generalizations (for a general $V$) of the classical down(up)-closed subsets construction that corresponds to the case in which $V= \mathbf{2}$.
		\item Given a $V$-category $(X,a)$, there are two $V$-functors, called the Yoneda embedding and the co-Yoneda embedding:
		$$ \Yo{X} : X \rightarrow \mathbb{D}(X), \ \ x \mapsto a(-,x),$$
		$$ \CoYo{X} : X \rightarrow \mathbb{U}(X), \ \ x \mapsto a(x,=).$$
		Moreover, one can prove that
		$$ \mathbb{U}(X)[\CoYo{X}(x),g] = g(x), \ \ \mathbb{D}(X)[\Yo{X}(x),g] = g(x).$$
		The last two results are known as the co-Yoneda lemma and Yoneda lemma, respectively. Notice that, for a general $X$, $\Yo{X}$ and $\CoYo{X}$ are not injective functions. They are injective iff $X$ is separated (see \cite[Proposition~1.5]{HT10}).
		\item \label{TensorMon} Let $(X,a)$ and $(Y,b)$ be $V$-categories. We define their tensor product
		$$ X \boxtimes Y = (X\times Y, a \otimes b).$$
		In particular, one has: $X \boxtimes K \simeq X$ where $K$ denotes the one-point $V$-category $(1,k)$.\par\medskip
		For $V = \mathbf{2}$, the ordered structure on $X \boxtimes Y$ is the product order. This means that $(x_1,y_1) \leq_{X \boxtimes Y} (x_2, y_2)$ if and only if $x_1 \leq_X x_2$ and $y_1 \leq_Y y_2$.\par\medskip
		For $V= [0, \infty]^{\op}$, the metric structure on $X \boxtimes Y$ is the taxicab metric, which is defined as:
		$$ d_{X \boxtimes Y }((x_1,y_1),(x_2, y_2) ) = d_X(x_1, x_2) + d_Y(y_1,y_2).$$
	\end{enumerate}

\end{Exs}
In this way we define $\Vcat$ as the category whose objects are $V$-categories and whose arrows are $V$-functors. Moreover, $\Vcat$ becomes an order-enriched category, if we define, for $V$-functors ${f,g : (X,a) \rightarrow (Y,b)}$,
$$ f \leq g \mbox{  whenever }  k \leq \bigwedge_{x\in X} b(f(x),g(x)).$$
With the tensor product previously defined, $\Vcat$ becomes a closed monoidal category, since one can show that, for $V$-categories $(X,a), (Y,b),(Z,c)$, one has
$$\Vcat(X \boxtimes Y,Z) \simeq \Vcat(X, [Y,Z]) \simeq \Vcat(Y, [X,Z]).$$
This allows us to define monoids with respect to such product, which we call monoidal $V$-categories.
\begin{defin}
A monoidal $V$-category $(X,a, \mult, u_X)$ is a $V$-category $(X,a)$ equipped with two $V$-functors: $\mult : X \boxtimes X \rightarrow X$ and $ u_X : K \rightarrow X$, such that $(X,a, \mult, u_X)$ is a monoid (with respect to the monoidal structure $(\boxtimes, K)$).
\end{defin}
\begin{re}
Notice that $K = (1,k)$ is a separator in $\Vcat$. This means that, for all pairs of parallel $V$-functors $f,g : X \rightarrow Y$ in $\Vcat$, if $f \cdot  z = g \cdot z$ for every $V$-functor $z : K \rightarrow X$, then $f = g$.
\end{re}
\begin{re}
	Notice that, for $V=\mathbf{2}$, a monoidal ordered set is just an ordered monoid. That is to say it is a monoid endowed with an order relation which is compatible with the monoid structure.\\
	For $V = [0, \infty]^{\op}$, a monoid in $ [0, \infty]^{\op} \mbox{-} \mathtt{Cat}$ is a metric space endowed with a monoid structure on its underlying set which is compatible with the metric. Examples of monoidal metric spaces are the underlying additive groups of normed vector spaces.
\end{re}
\subsection{Distributors and the Presheaf Monad}\label{Presheaf}
Bénabou introduced distributors in \cite{LD} and since then they played an important role in category theory. They can be seen as generalizations of ideal relations from order theory, that is to say, subsets of the cartesian product of ordered sets $X, Y$ which are downward closed in $X$ and upward closed in $Y$.\par\medskip
The presheaf construction is one of the cornerstones of category theory. A presheaf is a generalization of a downward closed subset on an ordered set $X$. In the following we are going to show how the presheaf construction is part of a monad defined on $\Vcat$.
\begin{defin}
Let $(X,a)$ and $ (Y,b)$ be $V$-categories. A $V$-\textit{distributor} (or simply a distributor) ${j  : (X,a) \xslashedrightarrow{} (Y,b)}$ is a $V$-relation between the underlying sets such that:
\begin{itemize}
	\item $j \dist a \leq j;$
	\item $b \dist j \leq j.$
\end{itemize}
\end{defin}
Since the composite of distributors is again a distributor, we define the quantaloid $\Dist{V}$ in the same way as we defined $ \Mat{V}.$ In $\Dist{V}(X,X)$ the $V$-structure $a$ plays the role of the identity, since for every distributor $j : X \xslashedrightarrow{} Y$, one has
$$ b \dist j = j \dist a = j.$$
Given a $V$-functor $f : (X,a) \rightarrow (Y,b)$, we define two arrows in $\Mat{V}$:
\begin{itemize}
\item $f_{*} : X \xslashedrightarrow{} Y,  \ \ f_{*}(x,y) = b(f(x),y);$
\item  $f^{*} : Y \xslashedrightarrow{} X, \ \ f^{*}(y,x) = b(y,f(x)).$
\end{itemize}
One has:
\begin{lem}
 The $V$-relations $f_{*}$ and $f^{*}$ are both distributors, moreover, $f_{*} \dashv f^{*}$ in $\Dist{V}.$
\end{lem}
In this way we have two $2$-functors
$$ (-)_{*} : \Vcat^{\mathtt{co}} \rightarrow \Dist{V}, \ \  (- )^{*}:  \Vcat \rightarrow \Dist{V}^{\op}.$$
By juggling with the definition of distributor, one can show that distributors between $V$-categories $(X,a), (Y,b)$ are in bijective correspondence with $V$-functors between $X^{\op} \boxtimes Y$ and $V$. It is easy to prove that this correspondence is functorial and that it gives an equivalence of ordered sets
$$\Dist{V}(X,Y) \simeq \Vcat(X^{\op} \boxtimes Y,V) \simeq \Vcat(Y,\mathbb{D}(X)),$$
where we associate to every $V$-distributor $j : X \xslashedrightarrow{} Y$ its \textit{mate}
$$ \ulcorner j \urcorner : Y \rightarrow \mathbb{D}(X), \ \ y \mapsto j(-,y).$$
\begin{prop}\label{DistAdj}
  The $2$-functor $(-)^* : \Vcat \rightarrow \Dist{V}^{\op}$ is left adjoint to the $2$-functor
  \[
  \begin{tikzcd}
  	\Dist{V}^{\op} \ar[r, "\mathbb{D}(-)"] & \Vcat, \mbox{ } Y \xslashedrightarrow{ } X \mapsto - \dist j : \mathbb{D}(X) \rightarrow \mathbb{D}(Y).
  \end{tikzcd}
  \]
\end{prop}
The $2$-monad induced by this $2$-adjunction has as underlying $2$-functor
$$\mathbb{D}(-): \Vcat \rightarrow \Vcat, \ \ f : X \rightarrow Y \mapsto \mathbb{D}(f) := - \dist f^{*} :\mathbb{D}(X) \rightarrow \mathbb{D}(Y)$$
and it has as unit, at a $V$-category $X$,
$$ \Yo{X} : X \rightarrow \mathbb{D}(X),$$
and as multiplication
$$- \dist (\Yo{X})_* : \mathbb{D}(X)^2 \rightarrow \mathbb{D}(X).$$
From the Yoneda lemma it follows that the monad $(\mathbb{D}(-),\Yo{-}, - \dist (\Yo{-})_*)$ is of \textit{Kock–Zöberlein} type (see \cite{KOCK199541}).\\
The $2$-category of pseudo-algebras for this monad is $2$-equivalent to the $2$-category formed by cocomplete $V$-categories and cocontinuous $V$-functors among them with the $2$-structure inherited by the one on $\Vcat$, and denoted by $\Coco{\Vcat}$. These last two observations, combined together, allow us to give a characterization of cocomplete $V$-categories.
\begin{teorema}\label{Coco}
	Let $(X,a)$ be a $V$-category. The following are equivalent:
  \begin{itemize}
    \item $(X,a)$ is a cocomplete $V$-category;
    \item There exists a $V$-functor
  	$$ \Sup_X : \mathbb{D}(X) \rightarrow X,$$
    such that, for every $x\in X$, $\Sup_X(\Yo{X}(x))\simeq x;$
    \item $(X,a)$ is \textit{pseudo-injective} with respect to fully faithful $V$-functors. That is to say, for every $V$-functor ${f: (Y,b) \rightarrow (X,a)}$ and for every fully faithful $V$-functor $i : (Y,b) \rightarrow (Z,c)$, there exists an extension ${f': (Z,c) \rightarrow (X,a)}$ such that $f' \cdot i \simeq f$.
  \end{itemize}
\end{teorema}
\begin{re}
Notice that, since $(\mathbb{D}(-),\Yo{-}, - \dist (\Yo{-})_*)$ is of Kock-Zöberlein type, $\Sup_X$ (whenever it exists) is automatically the left adjoint to the Yoneda functor.
\end{re}
A $V$-category $(X,a)$ is called \textit{separated} (see \cite{HT10}) whenever $f \simeq g$ implies $f = g$, for all $V$-functors of the form $f,g : (Y,b) \rightarrow (X,a)$. Separated cocomplete $V$-categories are strict algebras for the presheaf monad $(\mathbb{D}(-),\Yo{-}, - \dist (\Yo{-})_*)$.
\begin{teorema}\label{CocoSep}
	Let $(X,a)$ be a $V$-category. The following are equivalent:
  \begin{itemize}
    \item $(X,a)$ is a separated cocomplete $V$-category;
    \item There exists a $V$-functor
  	$$ \Sup_X : \mathbb{D}(X) \rightarrow X,$$
    such that, for every $x\in X$, $\Sup_X(\Yo{X}(x))=  x;$
    \item $(X,a)$ is \textit{injective} with respect to fully faithful $V$-functors. That is to say, for every $V$-functor ${f: (Y,b) \rightarrow (X,a)}$ and for every fully faithful $V$-functor $i : (Y,b) \rightarrow (Z,c)$, there exists an extension ${f': (Z,c) \rightarrow (X,a)}$ such that $f' \cdot i  = f$.
  \end{itemize}
\end{teorema}
Every set $X$ can be endowed with the discrete $V$-structure given by
$$
d_X(x_1,x_2) =
\begin{cases}
\perp &\mbox{ if } x_1 \neq x_2, \\
k &\mbox{ if } x_1 = x_2.
\end{cases}
$$
In this way we obtain a functor $d : \Sets \rightarrow \Vcats$, where the latter is the full subcategory of $\Vcat$ formed by separated $V$-categories. Since presheaf categories are always separated, we can compose it with
$$\mathbb{D}(-): \Vcats \rightarrow \Coco{\Vcats}.$$ In this way we get a functor which is left adjoint to the forgetful functor $\Coco{\Vcats} \rightarrow \Sets$.
\begin{teorema}\label{SetMon}
The forgetful functor $G:\Coco{\Vcats} \rightarrow \Sets$ is monadic.
\end{teorema}
\begin{proof}(Sketch)
	The proof relies on Duskin's criterion: we have to show that $G$ reflects isomorphisms and that $\Coco{\Vcats}$ has, and $G$ preserves, coequalizers of $G$-equivalence relations (see \cite{Duskin}).\par \medskip

	Let $f : (X,a) \rightarrow (Y,b)$ be in $\Coco{\Vcats}$ such that $f$ is a bijection with $g$ be its inverse (of course in $\Sets$). Let $- \dist f_{*} :\mathbb{D}(Y) \rightarrow \mathbb{D}(X) $ be the right adjoint to $\mathbb{D}(f)$. One can easily show that $g = \Sup_X \cdot (- \dist f_{*}) \cdot \Yo{Y}$ and thus that it is a $V$-functor.\par \medskip

	Let $R \rightrightarrows X$ be a $G$-equivalence relation. One can easily show that $(Q = X/R, a_R)$, where ${a_R(\overline{x},\overline{y}) = a(x,y)}$, is the coequalizer of $R \rightrightarrows X$ in $\Vcats$. Since $$\mathbb{D}(-): \Vcats \rightarrow \Coco{\Vcats}$$ preserves colimits (since it is a left adjoint), it follows that
	\[
	\begin{tikzcd}[row sep=large, column sep=large]
	  \mathbb{D}(R) \arrow[r, ,"\mathbb{D}(\pi_1)", shift left]
	        \arrow[r,"\mathbb{D}(\pi_2)", shift right, swap] & \mathbb{D}(X) \ar[r, "\mathbb{D}(\pi)"] &\mathbb{D}(Q)
	\end{tikzcd}
	\]
	is a coequalizer in $ \Coco{\Vcats}$ too. Moreover, one can show that there is a splitting given by $ {- \dist \pi_{*} : \mathbb{D}(Q)\rightarrow \mathbb{D}(X)}$, and by $- \dist \pi_{1*} : \mathbb{D}(X) \rightarrow \mathbb{D}(R).$ Moreover, since it splits, it remains a coequalizer also in $\Vcats$. We have the following commutative diagram
	\[
	\begin{tikzcd}[row sep=large, column sep=huge]
	 \mathbb{D}(R) \arrow[r, ,"\mathbb{D}(\pi_1)", shift left]
				 \arrow[r,"\mathbb{D}(\pi_2)", shift right, swap] \ar[d, "\Sup_R", swap] & \mathbb{D}(X) \ar[r, "\mathbb{D}(\pi)"]\arrow[l,  "- \dist \pi_{1*}", bend right, swap] \ar[d,, "\Sup_X", swap] & \mathbb{D}(Q) \arrow[l,  "- \dist \pi_{*}", bend right, swap] \arrow[d, "\exists !", dashed]\\
				 R \arrow[r, ,"\pi_1", shift left]
							 \arrow[r,"\pi_2", shift right, swap] & X \ar[r, "\pi", swap] & Q.
 \end{tikzcd}
 \]
Here the dashed arrow, which comes from the universal property of coequalizers, defines an algebra structure on $(Q,a_R)$ and it proves that it is cocomplete. In order to conclude, one has to prove that $(Q,a_R)$ is the coequalizer of $R \rightrightarrows X$ in $\Coco{\Vcats}$. Suppose $h : (X,a) \rightarrow (Y,b)$ is in $\Coco{\Vcats}$ and it is such that $h \cdot \pi_1 = h \cdot \pi_2$. By the universal property of coequalizers we get a unique arrow $f : \mathbb{D}(Q) \rightarrow (Y,b)$. By taking $f \cdot \Yo{Q} : (Q, a_R) \rightarrow (Y,b)$, one can prove the universal property also for
$(Q, a_R)$. \\
\end{proof}
Before we proved the previous theorem, we stated that the left adjoint to the forgetful functor
$$G:\Coco{\Vcats} \rightarrow \Sets$$
is the composite of $d : \Sets \rightarrow \Vcats$ with the presheaf functor $\mathbb{D}(-) : \Vcats \rightarrow \Coco{\Vcats}$. If we write it down, we can easily discover that this functor sends a set $X$ to the free cocomplete $V$-category whose underlying set is $\Pow{V}$, where $\Pow{V}(X) = V^X$. If we study the monad which arises from the adjunction $\Pow{V} \dashv G$, we discover that the resulting monad is the $V$-powerset monad $(\Pow{V}, u, n)$, the enriched generalization of the classical powerset monad, where
$\Pow{V} : \Sets \rightarrow \Sets$ is defined by putting $\Pow{V}(X) = V^X$ and, for $f: X \rightarrow Y$ and $\phi \in V^X$
$$ \Pow{V}(f)(\phi)(y) = \bigvee_{x \in f^{-1}(y) }\phi(x);$$
and:
\begin{itemize}
	\item $u_X : X \rightarrow V^X$ is the transpose of the diagonal $\bigtriangleup_X : X \times X \rightarrow V$;
	\item $n_X : \Pow{V}(\Pow{V}(X)) \rightarrow \Pow{V}(X)$ is defined by $n_X(\Phi)(x) = \bigvee_{\phi \in V^X} \Phi(\phi) \otimes \phi(x)$.
\end{itemize}
In this way we have the equivalence
$$\Coco{\Vcats} \simeq \Alg{V},$$
which can be explicitly described as the one that sends the cocomplete $V$-category $(X,a)$ to the algebra $(X, \alpha)$, where $\alpha(\psi)= \mathtt{Sup_X}(\psi \dist a)$.
\begin{re}
	Notice that the equivalence $\Coco{\Vcats} \simeq \Alg{V}$ generalizes the well known equivalence
	$$\mathtt{Sup} \simeq \Alg{\mathbf{2}},$$
	where $\Pow{\mathbf{2}}$ is the \textit{vanilla} powerset monad and $\mathtt{Sup}$ is the category of suplattices with suprema preserving maps.
\end{re}
\subsection{Enrichment via Actions}\label{ModAlg}
To every $V$-category $(X,a)$ we can associate an ordered set $(X, \leq_a)$, where the order is defined as
$$ x \leq_a y \iff k \leq a(x,y).$$
We call $(X, \leq_a)$ the underlying ordered set of the $V$-category $(X,a)$. This defines a $2$-functor
$$ \Vcat \rightarrow \mathtt{Ord}, \ \ f : (X,a)\rightarrow (Y,b) \mapsto f : (X, \leq_a) \rightarrow (Y,\leq_b).$$
\begin{re}
	Notice that the underlying orderd set of the $V$-category $(V,[-,=])$ is $(V,\leq)$, the underlying partially ordered set of the quantale $V$.
\end{re}
\begin{re}
	Notice that $ \Vcat \rightarrow \mathtt{Ord}$ restricts to a $2$-functor
	$$  \Vcats \rightarrow \mathtt{Ord}_{\mathtt{sep}}.$$
	Moreover, it is easy to see that $\mathtt{Ord}_{\mathtt{sep}} \simeq \mathtt{Pos}$, where the latter is the $2$-category of partially ordered sets and monotone maps.
\end{re}
\begin{re}
	The arguments we are going to use in this paper rely---mainly---on the monadicity over $\Sets$ of certain categories. For this reason we restrict ourself to consider only separated categories.
\end{re}
\begin{defin}
	We say that a $V$-category $(X,a)$ is copowered if, for all $x\in X$, $a(x, =) : X \rightarrow V$ admits a left adjoint in $\Vcat$ denoted by $- \odot x : V \rightarrow X$. That is to say
$$ a (u \odot x, y) = [u, a(x,y)],$$
for all $x,y \in X$ and $u \in V.$ We say that a $V$-functor $f: (X,a) \rightarrow (Y,b)$ between copowered $V$-categories preserves copowers if, for all $x\in X$ and $u \in V$, $f(u \odot x) \simeq u \odot f(x)$.
\end{defin}
In this way we can form the $2$-category of copowered categories with copowers preserving $V$-functors among them, denoted as $\Vcat^{\odot}$. In the same way, if we consider only separeted $V$-categories, we obtain the category $\Vcats^{\odot}$.\par\medskip
If we start with a separated copowered category $(X,a)$ and we take its underlying ordered set, then ${- \odot x : V \rightarrow X}$ becomes a monotone map of the type
$$- \odot x : (V, \leq) \rightarrow (X,\leq_a).$$
Moreover, we have the following lemma.
\begin{lem}\label{2f}
Under the same hypothesis as above, the monotone map
$$- \odot x : (V, \leq) \rightarrow (X,\leq_a),$$
enjoys the following properties, for all $x\in X $, $u,v \in V$:
\begin{itemize}
  \item $k \odot x = x;$
  \item $v \odot (u \odot x)  = (v \otimes u) \odot x;$
  \item $  ( \bigvee_i u_i ) \odot x = \bigvee_i (u_i \odot x)$, for every set $\{u_i \mbox{ | } i \in I\}$ of elements of $V$.
\end{itemize}
\end{lem}
\begin{proof}
  First observe that $k \odot x = x$ follows from $[k,w] = w$.\\
  Fix an $x\in X$. Then, for all $y\in X$, we have
  \begin{alignat*}{2}
    a( v \odot (u \odot x), y ) & = [v, a(u \odot x, y)] \\
    & = [v, [u, a(x,y)]] \\
    & = [v \otimes u, a(x,y)]\\
    & = a( (v \otimes u)\odot x, y ),
  \end{alignat*}
  from which $v \odot (u \odot x)  = (v \otimes u) \odot x$ follows.\\
  Finally, the last property follows from the adjunction $- \odot x \dashv a(x, =)$.\\
\end{proof}
\begin{defin}
Let $\mathtt{Pos}^V_{ \vee }$ be the category described as follows. An object of $\mathtt{Pos}^V_{ \vee }$ is a poset $(X,\leq_X)$ equipped with a monotone map
$$ \rho : V \boxtimes X \rightarrow X,$$
such that, for all $x\in X $, $u,v \in V$:
\begin{itemize}
  \item $\rho(k,x) = x;$
  \item $\rho(v, \rho(u,x)) = \rho( v \otimes u, x);$
  \item $\rho( \bigvee_i u_i,x ) = \bigvee_i \rho(u_i, x),$ for every set $\{u_i \mbox{ | } i \in I\}$ of elements of $V$.
\end{itemize}
An arrow $f: (X,\leq_X, \rho) \rightarrow (Y, \leq_Y, \theta)$ in $\mathtt{Pos}^V_{ \vee }$ is a monotone map between the underlying ordered sets $(X,\leq_X)$ and $(Y, \leq_Y)$, such that the following diagram commutes
\[
\begin{tikzcd}
V \boxtimes X \ar[d, "\rho", swap] \ar[r, "\Id \boxtimes f"] & V \boxtimes Y  \ar[d, "\theta"] \\
X \ar[r, "f"] & Y.
\end{tikzcd}
\]
\end{defin}
\begin{re}
	Let $X, Y$ be ordered sets. Then $X \Tensor{} Y \simeq X \times Y$, that is to say the monoidal structure $\Tensor{}$ in the category $\mathtt{Ord}$ coincides with the cartesian product $\times$.
\end{re}
\begin{re}
  Notice that $(V,\leq)$ acts on itself via the multiplication $ \otimes : V \times V \rightarrow V$. Moreover, since $\otimes$ preserves suprema, we also have
  $(\bigvee_i u_i) \otimes v = \bigvee_i (u_i \otimes v).$
\end{re}
\begin{prop}
  There exists a $2$-functor $\Vcats^{\odot} \rightarrow \mathtt{Pos}^V_{ \vee }$ that associates to a copowered $V$-category $(X,a)$ its underlying ordered sets $(X,\leq_a)$ with the action given by $ - \odot = : V \boxtimes X \rightarrow V$.
\end{prop}
\begin{re}\label{AFT}
  Notice that, from the adjoint functor theorem, it follows that
  $$\rho( \bigvee_i u_i , x) = \bigvee_i \rho(u_i, x)$$
  is equivalent to the statement: $\rho_x =  \rho(-, x) : V \rightarrow X$ has a right adjoint for all $x \in X$. In particular, when we apply this to $(V,\leq)$, we get as a right adjoint the internal hom $[x,=]$.
  This crucial observation will allow us to define a $V$-structure starting from the action.
\end{re}
Let $(X,\leq_X, \rho)$ be an object of $\mathtt{Pos}^V_{ \vee }$. By Remark \ref{AFT}, for all $x \in X$, there exists a monotone map $a(x,=): X \rightarrow V$ which is right adjoint to $\rho_x : V \rightarrow X$. Thus we have, for all $x,y \in X$ and $v \in V$,
$$\rho(v,x) \leq y \mbox{ $\iff$ } v \leq a(x,y).$$
In this way we can define a $V$-relation $a : X \xslashedrightarrow{} X$. As one might expect, this relation defines a $V$-structure on $X$.
\begin{lem}
  The $V$-relation $a : X \xslashedrightarrow{} X$ defines a $V$-structure on the set $X$. Moreover, if $$f: (X,\leq_X, \rho) \rightarrow (Y, \leq_Y, \theta)$$ is an arrow in $\mathtt{Pos}^V_{ \vee }$, then $f: (X,a) \rightarrow (Y,b)$ becomes a $V$-functor (where $a$ and $b$ are the $V$-structures induced by
$\rho$ and $\theta$).
\end{lem}
\begin{proof}
We have to show that, for all $x,y, z \in X$:
\begin{itemize}
  \item $ k\leq a(x,x);$
  \item $ a(x,y) \otimes a(y,z ) \leq a(x,z).$
\end{itemize}
The first one follows directly from $\rho(k,x) = x$, while the second one follows from $\rho(v,\rho(u,x)) = \rho( v \otimes u,x)$ and the from the adjunction $\rho_x \dashv a(x,=).$\par \medskip

Let $f: (X,\leq_X, \rho) \rightarrow (Y, \leq_Y, \theta)$ be an arrow in $\mathtt{Pos}^V_{ \vee }$ and call $b$ the $V$-structure induced by
$\theta$. Then, if we fix $x \in X$, the diagram
\[
\begin{tikzcd}
V \boxtimes X \ar[d, "\rho", swap] \ar[r, "\Id \boxtimes f"] & V \boxtimes X  \ar[d, "\theta"] \\
X \ar[r, "f"] & Y
\end{tikzcd}
\]
becomes
\[
\begin{tikzcd}
    & V \ar[dl, "\rho_x", swap] \ar[dr, "\theta_{f(x)}"] &   \\
     X \ar[rr, "f"] &  & Y.
\end{tikzcd}
\]
By general theory, if we take the right adjoints, the corresponding diagram does not commute anymore, but one has
\[
\begin{tikzcd}
    & V &   \\
     X \ar[ur, "\mathtt{a(x,=)}", ""{name=A, below}] \ar[rr, "f", swap] &  & Y \ar[ul, "\mathtt{b(f(x),=)}", ""{name=B,above}, swap]\ar[from=A, to=B, symbol= \leq]
\end{tikzcd}
\]
which, in pointwise terms, means that
$$ a(x,y) \leq b(f(x),f(y)).$$
Since we can vary both $x$ and $y$, this proves that $f$ is a $V$-functor.\\
\end{proof}
\begin{lem}
Let $(X,a)$ be the $V$-category obtained from an object $(X,\leq_X,\rho)$ of $\mathtt{Pos}^V_{ \vee }$. Then $(X,a)$ is a copowered category, with copowers given by $\rho_x$ for all $x \in X$.
\end{lem}
\begin{proof}
Since $2$-functors preserve adjoints and the $V$-structure on $V$ is the one induced by its multiplication, the result follows.\\
\end{proof}
We can define a $2$-functor
$$\mathtt{Pos}^V_{ \vee } \rightarrow \Vcats^{\odot}.$$
As expected we have the following result.
\begin{teorema}\label{Equi}
  The two $2$-functors
  \[
  \begin{tikzcd}
    \Vcats^{\odot} \ar[r, shift left] & \mathtt{Pos}^V_{ \vee } \ar[l, shift left]
  \end{tikzcd}
  \]
  establish a $2$-equivalence between $\Vcats^{\odot}$ and $\mathtt{Pos}^V_{ \vee }$.
\end{teorema}
\begin{proof}(Sketch)
Let $(X,a)$ be a separated copowered $V$-category. The indued $V$-structure
$\tilde{a}$ on the underlying ordered set $(X,\leq_a)$ is defined as, for $x,y \in X$ and $v\in V$,
$$ v \leq \tilde{a}(x,y) \mbox{ $ \iff$ } v \odot x \leq_a y,$$
where
$$ x \leq_a y \mbox{ $\iff$ } k \leq a(x,y).$$
Hence, for all $u \in V$,
\begin{alignat*}{2}
	u \leq \tilde{a}(x,y) & \mbox{ $\iff$ } u \odot x \leq_a y \\
  & \mbox{ $\iff$ } k \leq a(u \odot x, y)= [u, a(x,y)] \\
  & \mbox{ $\iff$ } u \leq a(x,y),
\end{alignat*}
which implies $\tilde{a}(x,y) = a(x,y).$
Consider $(X,\leq_X, \rho)$ in $\mathtt{Pos}^V_{ \vee }$. Then, by calling $\leq_a$ the underlying order structure of the induced $V$-category $(X,a)$, we have
\begin{alignat*}{2}
	x  \leq_a y  & \mbox{ $\iff$ } k \leq a(x,y) \\
  & \mbox{ $\iff$ } \rho(k,x) \leq_X y \\
  & \mbox{ $\iff$ } x \leq_X y.
\end{alignat*}
\end{proof}

We investigate now if we can further tune the $2$ equivalence
$$\Vcats^{\odot} \simeq \mathtt{Pos}^V_{ \vee }. \ \ \spadesuit$$
Remember that a sufficient and necessary condition for a $V$-category $(X,a)$ to be cocomplete is to be copowered and to have all conical suprema (see \cite{ECT}). In the light of this result, it is natural to ask if we can restrict $\spadesuit$ to a $2$-equivalence of the form:
$$\Coco{\Vcats} \simeq \Coco{\mathtt{Pos}}^V_{?},$$
where $?$ reflects the \textit{a priori} unknown property (or properties) that we have to add in order to obtain an equivalence.\par\medskip

Before we dip further into our quest, let us spend a few words about $\Coco{\mathtt{Pos}}$. We proved that
$$\Coco{\mathtt{Pos}} \simeq \mathtt{Sup} \simeq \Alg{\mathbf{2}}.$$
Here $\mathtt{Sup}$ denotes the $2$-category of suplattices with suprema preserving monotone maps among them, while $\Alg{\mathbf{2}}$ is the Eilenberg-Moore category for the powerset monad $\mathtt{P}_{\mathbf{2}}$. Since $\mathtt{P}_{\mathbf{2}}$ is a \textit{strong commutative monad}, $\mathtt{Sup}$ becomes a closed symmetric monoidal category (see Appendix \ref{AppendixI} or \cite{EGTG} for a more direct construction) $(\mathtt{Sup}, \Tensor{\mathbf{2}}, \mathbf{2})$ with the monoidal structure that classifies bimorphisms. Here a bimorphism in $\mathtt{Sup}$ is a monotone map of type $f : X \boxtimes Y \rightarrow Z$ such that $f$ preserves suprema separately in both variables and where $\boxtimes$ is the tensor product we defined in Example \ref{TensorMon} of Examples \ref{ExamplesVcat}.
We have the following lemma.
\begin{lem}\label{CocoSup}
  Let $(X,a)$ be a cocomplete separated $V$-category. Then its underlying ordered set $(X, \leq_a)$ is cocomplete. Moreover, if $f : (X,a) \rightarrow (Y,b)$ is a cocontinuous $V$-functor between cocomplete $V$-categories, then $f : (X, \leq_a) \rightarrow (Y, \leq_b)$ preserves suprema.
\end{lem}
In the light of what we wrote before, and because of the properties of arrows in $\mathtt{Pos}^V_{ \vee }$, the copower of a cocomplete separated $V$-category $(X,a)$ extends to a unique suprema preserving map
$$ (V,\leq) \Tensor{\mathbf{2}} (X,\leq_a) \rightarrow (X, \leq_a).$$
This shows that we have a $2$-functor
$$\Coco{\Vcats} \rightarrow \Mod,$$
where the latter is the category whose objects are suplattices $(X,\leq_X)$ endowed with an action $\rho: V \Tensor{\mathbf{2}} X \rightarrow X$ and whose arrows are suprema preserving equivariant monotone maps.
\begin{re}
  Notice that, since the monoidal structure on $\mathtt{Sup}$ classifies bimorphisms, we can freely curry any arrow (in $\mathtt{Sup}$) of the type
  $$ f : X\Tensor{\mathbf{2}} Y \rightarrow Z.$$
  That is to say, for any $x \in X$ (and similarly for any $y \in Y$), the curried version of $f$,
  $$ f_x : Y \rightarrow Z, \ \  f_x : y \mapsto f(x,y),$$
  is an arrow in $\mathtt{Sup}$.
\end{re}
In order to build a $2$-functor in the opposite direction,
$$\Mod \rightarrow \Coco{\Vcats},$$
we need the following result.
\begin{lem}
  Let $(X, \leq_X, \rho)$ be an object of $\Mod$. Then the corresponding (separated) $V$-category $(X,a)$ is cocomplete. Moreover, if $f : (X, \leq_X, \rho) \rightarrow (Y,\leq_Y,\theta)$ is an arrow in $\Mod$, then the corresponding $V$-functor between the associates $V$-categories $f : (X,a) \rightarrow (Y,b)$ is cocontinuous.
\end{lem}
With the aid of the two previous lemmas, we obtain the analouge of Theorem \ref{Equi}.
\begin{teorema}\label{Rest}
  The $2$-equivalence
  $$\Vcats^{\odot} \simeq \mathtt{Pos}^V_{ \vee }$$
  restricts to a $2$-equivalence
  $$ \Coco{\Vcats} \simeq \Mod .$$
\end{teorema}
\subsection{Monadicity}\label{MonadSup}
Before we end this section, we would like to point out a "nice" consequence of the result we have just proved: a characterization of the enriched power set monad $\mathtt{P}_V$ as the monad obtained by the composition of the \textit{vanilla} powerset monad $\mathtt{P}_{\mathbf{2}}$ with an "action" monad we are now going to describe.\par\medskip

As we mentioned before---and proved in Subsection \ref{CocoMonad}---there is an equivalence of categories
$$\Coco{\Vcats} \simeq \Alg{V},$$
where the latter is the Eilenberg-Moore category for the enriched powerset monad $\mathtt{P}_V$, from which it follows that $\Mod$ is monadic over $\Sets$.
We have the following commutative diagram
\[
\begin{tikzcd}
 \Mod \ar[r] \ar[dr] & \Alg{V} \ar[d] \\
              & \Sets
\end{tikzcd}
\qquad
\begin{tikzcd}
 (X,\leq_X, \rho) \ar[r, mapsto] \ar[dr, mapsto] & (X, a_X) \ar[d, mapsto] \\
              & X.
\end{tikzcd}
\]
Since we know that $\Mod \rightarrow \Alg{V}$ is an equivalence, and since the forgetful functor $\Alg{V} \rightarrow \Sets$ is monadic, by the commutativity of the previous diagram it follows that the forgetful functor $\Mod \rightarrow \Sets$ is monadic too. We can decompose the forgetful functor $\Mod \rightarrow \Sets$ as follows.
\[
\begin{tikzcd}
 \Mod \ar[r] \ar[d] & \Alg{V} \ar[d] \\
        \mathtt{Sup} \ar[r]     & \Sets
\end{tikzcd}
\qquad
\begin{tikzcd}
 (X,\leq_X, \rho) \ar[r, mapsto] \ar[d, mapsto] & (X, a_X) \ar[d, mapsto] \\
      (X,\leq_X)  \ar[r, mapsto]      & X
\end{tikzcd}
\]
Since $\mathtt{Sup} \rightarrow \Sets$ is the right adjoint to the powerset functor $$\mathtt{P}_{\mathbf{2}} : \Sets \rightarrow \mathtt{Sup}, \ \ X \mapsto (\mathtt{P}_{\mathbf{2}}(X), \subseteq),$$
in order to conclude we only need to provide a left adjoint to $\Mod \rightarrow \mathtt{Sup}$. Then by composing it with the powerset functor we would have our desired monad.\par \medskip

Let's reveal the identity of the butler in our detective story. For a suplattice $(X, \leq_X)$, we define an action on $V \Tensor{\mathbf{2}} X$ by:

$$
 V \Tensor{\mathbf{2}} V \Tensor{\mathbf{2}} X  \xrightarrow{ \otimes \Tensor{\mathbf{2}} \Id}   V \Tensor{\mathbf{2}} X,  \mbox{ ($ \otimes : V \Tensor{\mathbf{2}} V \rightarrow V $ is the multiplication of V). }
$$

The fact that it defines an action follows directly from the fact that $\otimes$  defines a monoid structure on $V$. Moreover, it is clear that if we have a morphism $f : X \rightarrow Y$, then $\Id \Tensor{\mathbf{2}} f$ defines an equivariant (with respect to the aforementioned action) arrow.
\begin{prop}
	Let $V \Tensor{\mathbf{2}} = : \mathtt{Sup} \rightarrow \Mod$ be the functor we described before. Then it is left adjoint to the forgetful functor  $U :\Mod \rightarrow \mathtt{Sup}.$
\end{prop}
\begin{proof}
	The unit at $(X,\leq_X)$ in $\mathtt{Sup}$ is given by
	$$ \eta_X : X \xrightarrow{\sim} \mathbf{2} \Tensor{\mathbf{2}} X \xrightarrow{} V \Tensor{\mathbf{2}} X, \quad x \mapsto k \Tensor{2} x, $$
	while the counit at $(Y,\leq_Y, \rho)$ in $\Mod$ is given by
	$$\epsilon_Y : V \Tensor{\mathbf{2}} Y \xrightarrow{\rho}   Y.$$
	The unit-counit equations are easily seen to be satisfied.\\
\end{proof}
\begin{re}
  We can easily prove that the aforementioned adjunction is monadic. It is straightforward to show that $U :\Mod \rightarrow \mathtt{Sup}$ reflects isomorphisms. Let $X \rightrightarrows Y$ be a $U$-split pair in $\Mod$. Let $Z$ be the coequalizer of this $U$-split pair in $\mathtt{Sup}$, then
  \[
  \begin{tikzcd}
   V \Tensor{\mathbf{2}} X \arrow[r, shift left]
   \arrow[r, shift right] & V \Tensor{\mathbf{2}} Y \arrow[l, bend right] \arrow[r] &  V \Tensor{\mathbf{2}} Z \arrow[l,  bend right]
  \end{tikzcd}
  \]
  \\
  is a split coequalizer in $\mathtt{Sup}$. By using the split and the universal property of coequalizers, we have a unique arrow $V \Tensor{\mathbf{2}} Z \rightarrow Z$, as depicted in the diagram
  \begin{center}
  \begin{tikzcd}
  V \Tensor{\mathbf{2}} X \arrow[r, shift left]
  \arrow[r, shift right] \arrow[d]
  & V \Tensor{\mathbf{2}} Y \arrow[r] \arrow[d ]  & V \Tensor{\mathbf{2}} Z \arrow[d, "\exists !", dashed] \\
  X \arrow[r, shift left]
  \arrow[r, shift right]
  & Y\arrow[r]  & Z.
  \end{tikzcd}
\end{center}
  Since $ Y \rightarrow Z$ is an epimorphism, and since $\Tensor{\mathbf{2}}$ preserves epimorphisms, the action $V \Tensor{\mathbf{2}} Z \rightarrow Z$ makes $Z$ an object in $\Mod$ and thus $ Y \rightarrow Z$ is the coequalizer of $X \rightrightarrows Y$ in $\Mod$.\\
	The resulting monad $T =  (V \Tensor{\mathbf{2}} =, \eta, \mu)$, has as unit
  $$ \eta_X : X \xrightarrow{\sim} 1 \Tensor{\mathbf{2}} X \rightarrow V \Tensor{\mathbf{2}} M, $$
  and as multiplication
$$
  \mu_X :  V\Tensor{\mathbf{2}} V \Tensor{\mathbf{2}} X \xrightarrow{\otimes \Tensor{\mathbf{2}}  \Id}   V \Tensor{\mathbf{2}} X.
$$
\end{re}

If we compose $\mathtt{P}_{\mathbf{2}}$ with $V \Tensor{\mathbf{2}} =$, we have the left adjoint to the monadic forgetful functor $\Mod \rightarrow \Sets$ we were looking for. In this way we obtain a $\Sets$ monad $V \Tensor{\mathbf{2}} \mathtt{P}_{\mathbf{2}}(=)$ that sends a set $X$ to the underlying set of $(V, \leq) \Tensor{\mathbf{2}} (\mathtt{P}_{\mathbf{2}}(X), \subseteq)$ and a function $f : X \rightarrow Y$ to $\Id \Tensor{\mathbf{2}}\mathtt{P}_{\mathbf{2}}(f)$.\\
Hence we have an equivalence of monads:
$$ \mathtt{P}_V \simeq V \Tensor{\mathbf{2}} \mathtt{P}_{\mathbf{2}}.$$
\section{Quantale-Enriched Multicategories}
$(L,V)$-categories are a special case of the more general $(T,V)$-categories, where the list monad $L$ is considered. They are also the order-enriched version of \textit{multicategories} (see \cite{leinster2004higher} and \cite{hofmann2014monoidal} for an account on them, and \cite{Lambek} for a historical perspective). The basic idea is that, instead of having arrows with just a single object as the domain, we allow them to have as domain a list of objects.\par\medskip
In this section we introduce $(L,V)$-categories and some of their basic constructions, by mirroring what we have done in the previous section.
\subsection{(L,V)-Categories and (L,V)-functors}
Recall that the list monad is the monad whose underlying functor is given by
$$L : \mathtt{Sets} \rightarrow \mathtt{Sets}, \ \ f: X \rightarrow Y \mapsto Lf : \amalg_{n \geq 0} X^n \rightarrow \amalg_{m \geq 0} Y^m , \ \ \underline{x} =(x_1,..., x_n) \mapsto (f(x_1),...,f(x_n)),$$
and whose unit and multiplication at a set $X$ are defined as:
\begin{itemize}
	\item $e_X : X \rightarrow L(X),  \ \ x \mapsto (x);$
	\item $m_X : L^2(X) \rightarrow L(X), \ \ (\underline{x}_1,...,\underline{x}_n) \mapsto (x_{11}, ..., x_{1k}, ..., x_{n1}, ...,x_{nl}).$
\end{itemize}

\begin{re}
	Let $\underline{x}$, $\underline{w}$ be lists. In order to avoid possible confusion with the list of lists $\underline{\underline{y}} = (\underline{x},\underline{w})$, we denote the list obtained by concatenating $\underline{x}$ and $\underline{w}$ as $(\underline{x};\underline{w})$. Moreover, in the case in which one of the two is a single element list, we use the shortcut $(\underline{x};w)$ instead of $(\underline{x};(w)).$
\end{re}

We can extend (in a functorial way) the list monad $L$ to $\Mat{V}$ by defining, for $r: X \xslashedrightarrow{} Y$:
$$\tilde{L}r : L(X)  \xslashedrightarrow{} L(Y),  \ \ (\underline{x}, \underline{y}) \mapsto \begin{cases}
r(x_1,y_1) \otimes ... \otimes r(x_n, y_n) &\mbox{ if the two lists have the same length,} \\
\perp &\mbox{ otherwise.} \\
\end{cases} $$
One can prove that this particular extension defines a monad on $\Mat{V}$ that, moreover, preserves the involution
$$(-)^{\circ} : \Mat{V}^{\op} \rightarrow \Mat{V}.$$
\begin{re}
	From now on we will use $L$ for both the ordinary list monad and its extension to $\Mat{V}$.
\end{re}
This allows us to define the order-enriched category $\Mat{(L,V)}$ in which a morphism $r: X \kmodto Y$ is a $V$-relation of the form
$$ r : L(X) \xslashedrightarrow{} Y,$$
and in which composition is given by
$$ s \ldist r = s \dist Lr \dist m_X^{\circ},$$
where $e_X^{\circ} : X \kmodto X$ is the identity.
\begin{re}\label{K}
	Note that, due to the Kleisli-style composition we defined, $ - \ldist r$ preserves suprema, but $s \ldist (=)$ does not in general.
\end{re}
\begin{defin}
		An $(L,V)\mbox{-}$category is a pair $(X, a)$, where $X$ is a set and $a : X \kmodto X$ is an $(L,V)$-relation that satisfies:
	\begin{itemize}
		\item $e_X^{\circ} \leq a$;
		\item $a \ldist a \leq a.$
	\end{itemize}
\end{defin}
\begin{re}
	When $V = \mathbf{2}$, the $(L,\mathbf{2})$-structure of an $(L,\mathbf{2})$-category $(X,a)$ is a subset $a \subseteq L(X) \times X$ such that:
	\begin{itemize}
		\item for all $x \in X$, $ ((x),x) \in a$;
		\item given $(\underline{z}_1,...,\underline{z}_n) \in L^2(X)$, $\underline{x} \in LX$, and $y \in X$, such that
		$$((\underline{z}_1,...,\underline{z}_n),\underline{x} )\in La, \mbox{ and } (\underline{x}, y) \in a,$$
		then
		$$((\underline{z}_1;...;\underline{z}_n),y) \in a.$$
	\end{itemize}
	Notice that $La$ is the subset that corresponds to the relation $La : L^2(X) \times LX \rightarrow \mathbf{2}$ that one obtains by applying the extension of the list monad to the relation $a : LX \times X \rightarrow \mathbf{2}$.
\end{re}
\begin{defin}
	Let $(X,a)$ and $ (Y,b)$ be $(L,V)$-categories. An \textit{$(L,V)\mbox{-}$functor} $f: (X,a) \rightarrow (Y,b)$ is a function between the underlying sets such that
	$$ a \leq f^{\circ}\dist b \dist Lf,$$
	which, in pointwise terms, means that, for all $\underline{x} \in LX$, $y\in X$,
	$$ a(\underline{x},y) \leq b(Lf(\underline{x}), f(y)).$$
	If the equality holds, we call $f$ fully faithful.
\end{defin}
\begin{re}
If $V=\mathbf{2}$, then an $(L,\mathbf{2})$-functor $f: (X,a) \rightarrow (Y,b)$ satisfies, for all $\underline{x} \in LX$, $y\in X$,
$$ (\underline{x},y) \in a \mbox{ implies } (Lf(\underline{x}), f(y)) \in b.$$
Notice how this generalizes the classical monotonicity condition.
\end{re}
In this way we define $\Multi{V}$ as the category whose objects are $(L,V)$-categories and whose arrows are $(L,V)$-functors, moreover, $\Multi{V}$ becomes an order-enriched category if we define, for two $(L,V)$-functors $f,g : (X,a) \rightarrow (Y,b)$,
$$ f \leq g \mbox{  whenever }  k \leq \bigwedge_{x \in X} b(Lf((x) ),g(x)).$$
\begin{Exs}
	\begin{enumerate}
		\item Every set $X$ defines an $(L,V)$-category with $e^{\circ}_X$ as $(L,V)$-structure. In particular, we define the one-point $(L,V)$-category $E  = (1, e^{\circ}_{1} ).$
		\item Every set $X$ defines an $(L,V)$-category if we consider the free $L$-algebra on $X$, $(LX, m_X)$.
		\item $V$ itself defines an $(L,V)$-category where $[\underline{v}, w] = [v_1 \otimes ... \otimes v_n, w].$
		\item Let $(X,a)$ and $ (Y,b)$ be $(L,V)$-categories. We can form their tensor product $X \boxtimes Y  = (X \times Y, a \boxtimes b),$ where
		$$ a \boxtimes b (\gamma, (x,y)) = a(L\pi_1(\gamma), x) \otimes b(L\pi_2(\gamma), y).$$
		Here $\gamma \in L(X \times Y)$ and $\pi_1 ,\pi_2$ are the obvious projections. Unluckily, in general it is not true that $X \boxtimes E \simeq X.$
	\end{enumerate}
\end{Exs}
\begin{re}
	\label{Mon} In general, every monoidal $V$-category $(X,a, \mult, u_X)$ defines an $(L,V)$-category, where $${a(\underline{x}, y) = a(x_1 \mult ... \mult x_n,y).}$$\footnote{In particular $a((-), y) = a(u_X, y)$.} $(L,V)$-categories defined in this way are called \textit{representable} and their $(L,V)$-structure will be denoted by $ \hat{a} = a \dist \alpha$, where $$\alpha : L(X) \rightarrow X, \ \ \underline{x} \mapsto x_1 \mult ... \mult x_n, \quad (-) \mapsto u_X.$$
	In this way we can define a $2$-functor $ \mathtt{Kmp} :\Vcat^L \rightarrow \Multi{V}$ which has a left adjoint ${M : \Multi{V} \rightarrow \Vcat}$ that sends an $(L,V)$-category $(X,a)$ to $(LX, La \dist m_X^{\circ}, m_X)$ and an $(L,V)$-functor $f$ to $Lf$. $M$ is also a $2$-functor.\\
	Using the aforementioned adjunction, we can extend the monad $L$ to a monad on $\Multi{V}$, denoted by $L$ as well. Moreover, one can prove (see \cite{Chikhladze2015}) that there is an equivalence
	$$\Vcat^L \simeq \Multi{V}^L.$$
\end{re}
\begin{re}
A priori, due to the non-symmetric form of arrows in $\Mat{(L,V)}$, it is not clear how to define an $(L,V)$-category that seems to play the role of a dual. Luckily, we can use the adjunction $\mathtt{Kmp} \dashv M$ and the involution in $\Mat{V}$ to define, for an $(L,V)$-category $(X,a)$, its opposite category as $X^{\op} =(LX, m_X \dist L a^{\circ} \dist m_X)$. At first this might be seen as an \textit{ad hoc} definition, but if we apply this construction to a $V$-category $(X,a)$, seen as an $(L,V)$-category $(LX,e_X^{\circ} \dist a)$, we get
	$$X^{\op} = \mathtt{Kmp}(LX, L a^{\circ}),$$
	where $(LX, L a^{\circ})$ is the dual, as a $V$-category, of $(LX, La)$.
\end{re}
For any $(L,V)$-category $(X,a)$ we can form the $(L,V)$-category $\mathbb{D}_L(X)[-,=]$ whose underlying set consists of all $(L,V)$-functors of the form: $f : X^{\op} \boxtimes E\rightarrow V$ and whose $(L,V)$-structure is given by \label{DLstructure}
$$ \mathbb{D}_L(X)[\underline{f},g] =  \bigwedge_{(\underline{x}_1, ...,\underline{x}_n) \in LX^2 } [(f_1(\underline{x}_1),..., f_n(\underline{x}_n)), g(m_x((\underline{x}_1, ...,\underline{x}_n))))],$$
where $\underline{f} \in L( \mathbb{D}_L(X)) $ and $g \in  \mathbb{D}_L(X).$
\begin{re}\label{MultiYoneda}
We have a fully faithful functor, called the Yoneda embedding,
$$
\Yo{X} : X \rightarrow \mathbb{D}_L(X), \ \ x \mapsto a(-, x).
$$
Moreover, it can be proved that
$$ \mathbb{D}_L(X)[L\Yo{X}(\underline{x}),g] = g(\underline{x}).$$
The last result is known as the Yoneda Lemma.
\end{re}
\subsection{Distributors and the Presheaf Monad}
The relational point of view we used for introducing $V$-categories allows us to provide the corresponding notion of distributor for $(L,V)$-categories by considering the composition $\ldist$ defined in the previous section.
\begin{defin}\cite{CH09}
	Let $(X,a)$ and $ (Y,b)$ be $(L,V)$-categories. An \textit{$(L,V)$-distributor} $j  : (X,a) \kmodto (Y,b)$ is an $(L,V)$-relation between the underlying sets such that:
	\begin{itemize}
		\item $j \ldist a \leq j;$
		\item $b \ldist  j \leq j.$
	\end{itemize}
\end{defin}
Just as in the $V$-case, we define an order-enriched category $\Dist{(L,V)}$, where the composition is the one defined in $\Mat{(L,V)}$.
\begin{re}\label{mate}
	As in the $V$-case, one can prove
	$$\Dist{(L,V)}(X,Y) \simeq \Multi{V}(X^{\op} \boxtimes Y,V)\simeq \Multi{V}(Y,\mathbb{D}_L(X)).$$
	In particular, to every $(L,V)$-distributor $j : X \kmodto Y$ we can associate its \textit{mate}
	$$ \ulcorner j \urcorner : Y \rightarrow \mathbb{D}(X), \ \ y \mapsto j(-,y).$$
\end{re}
Just as in the $V$-case, for an $(L,V)\mbox{-}$functor $f: (X,a) \rightarrow (Y,b)$, we have two associated adjoint $(L,V)$-distributors:
\begin{enumerate}
	\item $f_{\circledast} : X \kmodto Y, \ \ f_{\circledast}(\underline{x},y) = b(Lf(\underline{x}), y);$
	\item $f^{\circledast} : Y \kmodto X, \ \ f^{\circledast}(\underline{y},x) = b(\underline{y},f(x)).$
\end{enumerate}
In this way we have two $2$-functors
$$ (-)_{\circledast} : \Multi{V}^{\mathtt{co}} \rightarrow \Dist{(L,V)}, \ \  (- )^{\circledast}:  \Multi{V} \rightarrow \Dist{(L,V)}^{\op}.$$
\begin{prop}
  The $2$-functor $(- )^{\circledast}:  \Multi{V} \rightarrow \Dist{(L,V)}^{\op}$ is left adjoint to the $2$-functor
  \[
  \begin{tikzcd}
  	\Dist{(L,V)}^{\op} \ar[r, "\mathbb{D}_L(-)"] & \Multi{V}, \mbox{ } Y \kmodto X \mapsto - \ldist j : \mathbb{D}_L(X) \rightarrow \mathbb{D}_L(Y).
  \end{tikzcd}
  \]
\end{prop}
The $2$-monad induced by this $2$-adjunction has as underlying $2$-functor
$$\mathbb{D}_L(-): \Multi{V} \rightarrow \Multi{V}, \quad f : X \rightarrow Y \mapsto \mathbb{D}_L(f) := - \ldist f^{\circledast} :\mathbb{D}_L(X) \rightarrow \mathbb{D}_L(Y).$$
It has as unit at $X$,
$$  \Yo{X} : X \rightarrow \mathbb{D}_L(X),$$
and as multiplication
$$  - \ldist (\Yo{X})_{\circledast} : \mathbb{D}_L(X)^2 \rightarrow \mathbb{D}_L(X).$$
In particular, as in the $V$-case, from the Yoneda lemma it follows that $(\mathbb{D}_L(-),\Yo{X},- \ldist (\Yo{X})_{\circledast})$ is of Kock–Zöberlein type.\par\medskip
Similarly to what happens in the $V$-case, one can prove that (pseudo)-algebras for the monad $\mathbb{D}_L$ are exactly cocomplete categories (see \cite{HOFMANN2011283} for more details). Moreover, since $(\mathbb{D}_L(-),\Yo{-},- \ldist (\Yo{-})_{\circledast})$ is of Kock–Zöberlein type, we have the analogon of Theorem \ref{Coco}(see \cite{HOFMANN2011283}).
\begin{teorema}
	Let $(X,a)$ be an $(L,V)$-category. The following are equivalent:
  \begin{itemize}
    \item $(X,a)$ is a cocomplete $(L,V)$-category;
    \item There exists an $(L,V)$-functor
  	$$ \Sup_X : \mathbb{D}_L(X) \rightarrow X,$$
    such that, for every $x\in X$, $\Sup_X(\Yo{X}(x))\simeq x;$
    \item $(X,a)$ is pseudo-injective with respect to fully faithful $(L,V)$-functors. That is to say, for every $(L,V)$-functor $f: (Y,b) \rightarrow (X,a)$ and for every fully faithful $(L,V)$-functor $i : (Y,b) \rightarrow (Z,c)$, there exists an extension $f': (Z,c) \rightarrow (X,a)$ such that $f' \cdot i \simeq f$.
  \end{itemize}
\end{teorema}
\subsection{Monadicity over $\Sets$}\label{CocoMonad}
As in the $V$-case, we restrict ourself to consider only separated $(L,V)$-categories. An $(L,V)$-category $(X,a)$ is called separated (see \cite{HT10}) whenever $f \simeq g$ implies $f = g$, for all $(L,V)$-functors of the form ${f,g : (Y,b) \rightarrow (X,a)}$. We have the analogue of Theorem \ref{CocoSep}.
\begin{teorema}
	Let $(X,a)$ be an $(L,V)$-category. The following are equivalent:
  \begin{itemize}
    \item $(X,a)$ is a separated cocomplete $(L,V)$-category;
    \item There exists an $(L,V)$-functor
  	$$ \Sup_X : \mathbb{D}_L(X) \rightarrow X,$$
    such that, for every $x\in X$, $\Sup_X(\Yo{X}(x)) =  x;$
    \item $(X,a)$ is injective with respect to fully faithful $(L,V)$-functors. That is to say, for every $(L,V)$-functor $f: (Y,b) \rightarrow (X,a)$ and for every fully faithful $(L,V)$-functor $i : (Y,b) \rightarrow (Z,c)$, there exists an extension $f': (Z,c) \rightarrow (X,a)$ such that $f' \cdot i = f$.
  \end{itemize}
\end{teorema}
One can prove that the forgetful functor $\Multi{V} \rightarrow \Sets$ has a left adjoint given by
$$ d : \Sets \rightarrow \Multi{V}, \mbox{  } X \mapsto (X, e^{\circ}_X).$$
Thus, the forgetful functor $$\Coco{\Multi{V}_{\mathtt{sep}}} \rightarrow \Sets,$$
where $\Coco{\Multi{V}_{\mathtt{sep}}}$ denotes the ($2$-)category formed by cocomplete separated $(L,V)$-category and cocontinuous $(L,V)$-functors among them, has a left adjoint which is given by the composite
$$
	 \Sets \xrightarrow{d}  \Multi{V}_{\mathtt{sep}} \xrightarrow{\mathbb{D}_L(-)}  \Coco{\Multi{V}_{\mathtt{sep}}}.
$$
As in the $V$-case we have (see Theorem $2.23$ of \cite{HOFMANN2011283}) that this functor is monadic.
\begin{teorema}
The forgetful functor $G: \Coco{\Multi{V}_{\mathtt{sep}}} \rightarrow \Sets$ is monadic.
\end{teorema}
Define $\Pow{L} =  \mathbb{D}_L \cdot d$. Then the monad which arises from the previous theorem is the monad $(\Pow{L}, e, n)$ whose unit and multiplication, at a set $X$, are given by
$$ e_X : X \rightarrow \Pow{L}(X), \mbox{ } x \mapsto \Yo{X}(x) = e^{\circ}_X(-,x),$$
$$ n_X =  - \ldist (\Yo{X})_{\circledast} : \Pow{L}\Pow{L}(X) \rightarrow \Pow{L}(X).$$
In Section \ref{InjMon} we will study better this monad and its algebras.
\section{Refining $\Mod \simeq \Alg{V}$}
In Section \ref{ModAlg} we proved that there exists an equivalence
$$\Mod \simeq \Alg{V}.$$
In this section, first we prove that both categories can be equipped with a monoidal structure, and then we prove that the aforementioned equivalence extends to the corresponding categories of monoids.\par \medskip
In Proposition \ref{StrongPow} we prove that $\Alg{V}$ admits a closed symmetric monoidal structure $(\Tensor{\Pow{V}},V)$ for which the free functor
$$\Pow{V} : \Sets \rightarrow \Alg{V}$$
becomes strong monoidal. This monoidal structure comes from the fact that the $V$-powerset monad $(\Pow{V}, u, n)$ is a strong commutative monad. Moreover, this monoidal structure has another interesting property: it classifies bimorphisms. This means that there exists a natural isomorphism (for all $(X,\alpha)$, $(Y, \beta)$, $(Z,\theta) \in \Alg{V}$)
$$\Bim_{\Alg{V}}( X \times Y, Z ) \simeq \Alg{V}(X \Tensor{\Pow{V}} Y, Z),$$
where
$$\Bim_{\Alg{V}}( X \times Y, = ) : \Alg{V} \rightarrow \Sets$$
is the functor that sends an algebra $(Z,\theta)$, to the set of bimorphisms of the form
$f : X \times Y \rightarrow Z$. Here a function $f : X \times Y \rightarrow Z,$ is a bimorphism if, for all $x \in X$, $y\in Y$,
$$f_x : Y \rightarrow Z, \mbox{ } y\mapsto f(x,y) \mbox{ and }f_y : X \rightarrow Z, \mbox{ } x\mapsto f(x,y)$$
are morphisms in $\Alg{V}$ (see Definition \ref{BimDef2} and Proposition \ref{EquiBimo}).
\begin{re}
	Notice that in the case in which $V = \mathbf{2}$ one obtains the well known monoidal structure on suplattices. See \cite{EGTG} for its description.
\end{re}

In \cite{EGTG}, Joyal and Tierney defined the tensor product $X \Tensor{V} Y$ of $V$-modules $(X, \leq_X, \rho)$ and $(Y, \leq_Y, \tau)$ as the coequalizer of

\[
\xymatrix@R+1pc@C+2pc{
	V \Tensor{\mathbf{2}} X \Tensor{\mathbf{2}} Y \ar@<-.5ex>[r]_{\rho \Tensor{\mathbf{2}} \Id} \ar@<.5ex>[r]^{(\Id \Tensor{\mathbf{2}} \gamma) \cdot \tau_{V,X}} & X \Tensor{\mathbf{2}} Y. & \mbox{ (where $\tau_{V,M} : V \Tensor{\mathbf{2}}  X \simeq X \Tensor{\mathbf{2}} V$)}
}
\]
With this tensor product one can prove that $\Mod$ becomes a symmetric closed monoidal category. Moreover, $\Tensor{V}$ classifies bimorphisms, where a bimorphism $f : X \times Y \rightarrow Z$ between $V$-modules is function such that, for all $x \in X$, $y\in Y$,
$$f_x : Y \rightarrow Z, \mbox{ } y\mapsto f(x,y) \mbox{ and }f_y : X \rightarrow Z, \mbox{ } x\mapsto f(x,y)$$
are morphisms in $\Mod$.\par \medskip

Let $f : X \times Y \rightarrow Z,$ be a bimorphism between $V$-modules. Since for all $x \in X$, $y\in Y$,
$$f_x : Y \rightarrow Z, \mbox{ } y \mapsto f(x,y) \mbox{ and }f_y : X \rightarrow Z, \mbox{ } x\mapsto f(x,y)$$
are morphisms in $\Mod$, by applying the forgetful functor $\Mod \rightarrow \mathtt{Sup}$, we get two morphisms in $\mathtt{Sup}$. Since the monoidal structure on $\mathtt{Sup}$ classifies bimorphisms too, we get a unique morphism $\overline{f}: X \Tensor{\mathbf{2}} Y \rightarrow Z$ in $\mathtt{Sup}$ that makes the following diagram commutes
\[
\begin{tikzcd}
	X \times Y \ar[r] \ar[rd, "f",swap] & X \Tensor{\mathbf{2}} Y \ar[d, "\overline{f}"] \\
	& Z.
\end{tikzcd}
\]
This shows that we can define a bimorphism in $\Mod$ in two equivalent ways:
\begin{itemize}
	\item As a function $f : X \times Y \rightarrow Z,$ such that $f$ is a morphism in $\Mod$ in each variable separately;
	\item A suprema preserving map $\overline{f}: X \Tensor{\mathbf{2}} Y \rightarrow Z$ such that its associate arrow $f : X \boxtimes Y \rightarrow Z,$ is equivariant in each variable separately.
\end{itemize}
In other words, we have the following natual bijections:
$$\Bim_{\Mod}(X \times Y, Z ) \simeq \Bim_{\Mod}(X \Tensor{\mathbf{2}} Y, Z ) \simeq \Mod(X \Tensor{V} Y, Z).$$
\begin{re}\label{StrongMod}
It is possible to show that a strong commutative monad $T$ induces the monoidal structure on $\Mod$ we have described. This monad is the monad introduced in Section \ref{MonadSup}: the monad ${T =  (V \Tensor{\mathbf{2}} =, \eta, \mu)}$. \\
We define our wannabe strenght as
$$
t_{X,Y} :  X \Tensor{\mathbf{2}} TY \xrightarrow{\tau_{X, V}\Tensor{\mathbf{2}} \Id}  T(X \Tensor{\mathbf{2}} Y),
$$
from which we can deduce that the co-strenght is given by
$$
	t'_{M,N} : TX \Tensor{\mathbf{2}} Y \xrightarrow{\tau_{TX, Y} } Y \Tensor{\mathbf{2}} TX \xrightarrow{\tau_{X, V} \Tensor{\mathbf{2}} \tau_{Y, X}}  T(X \Tensor{\mathbf{2}} Y).
	$$
It is pretty straightforward to show that $T$ is a strong monad; moreover, because $V$ is a commutative quantale, the fact that $T$ is commutative follows directly.
\end{re}
\begin{re}
	Notice that from Remark \ref{Unitor}, it follows that the left unitor $ l_M : V \Tensor{V} X \simeq X$, of the monoidal structure on $\Mod$, is the unique morphism associate to the bimorphism $\rho : V \Tensor{\mathbf{2}} X \rightarrow X $, where $(X, \leq_X, \rho)$ is in $\Mod$.
\end{re}
\begin{prop}\label{MonBim}
The equivalence
$$\Mod \simeq \Alg{V}$$
extends to an equivalence beetween the corresponding category of monoids
$$\Mon{\Mod}{V} \simeq \Mon{\Alg{V}}{\Pow{V}}.$$
\end{prop}
\begin{proof}
Let us call $[- ] : \Mod \rightarrow \Alg{V}$ the functor which realizes the equivalence $\Mod \simeq \Alg{V}$ which is the composite of $\Mod \simeq \Coco{\Vcats} $ we obtained in Theorem \ref{Rest} with $\Coco{\Vcats} \simeq \Alg{V}$ we described in Subsection \ref{Presheaf}. Let $f: X \times Y \rightarrow Z$ be a bimorphism between $V$-modules. Since equivalences preserve products we have $[X \times Y] \simeq [X]\times [Y]$. Moreover, because for all $x,y \in X,Y$, $f_x, f_y$ are morphisms in $\Mod$, we have $[f_x]$ and $[f_y]$ are morphisms in $\Alg{V}$. From $[X \times Y] \simeq [X]\times [Y]$  it follows that $[f_x] = [f]_x$ and $[f_y] = [f]_y$. This implies that $f$ defines a bimorphism $f : [X]\times [Y] \rightarrow [Z]$ in  $\Alg{V}$. This shows that we have a bijection
$$\Bim_{\Mod}( X \times Y, Z ) \simeq \Bim_{\Alg{V}}([X]\times [Y], [Z])$$
which can be easily seen to be natual (since $[-]$ is a functor and all the "change of base" components are given by pre-post composition).
Now, from
$$\Mod(X \Tensor{V} Y, Z) \simeq \Alg{V}([X \Tensor{V} Y ] ,[Z])$$
and from
\begin{equation} \label{equivalence}
\Bim_{\Mod}( X \times Y, Z ) \simeq \Bim_{\Alg{V}}([X]\times [Y], [Z]),
\end{equation}
it follows that
$$\Alg{V}([X] \Tensor{\Pow{V}} [Y], [Z] ) \simeq \Alg{V}([X \Tensor{V} Y ], [Z]),$$
from which we can deduce that
$$ [X] \Tensor{\Pow{V}} [Y] \simeq [X \Tensor{V} Y ].$$
The compatibility of $[- ]$ and the associators follows from the bijection \ref{equivalence} and from the fact that both associators derive from the associator of the cartesian product $\Sets$.\\
Moreover, since
$[V] \simeq V,$
and the unitors are compatible with $[- ]$, the result follows (see the next remarks for further details).\\
\end{proof}
\begin{re}
	Remember that $V$, as a $\Pow{V}$-algebra, has the structure given by
	$$ n_1 : \Pow{V}(\Pow{V}(1)) \rightarrow \Pow{V}(1), \ \ j \mapsto \bigvee_w j(w)\otimes w,$$
	while the $\Pow{V}$-algebra associated to the cocomplete $V$-category $(V, [-,=])$ is $(V, \alpha)$, where
	$$ \alpha(j)=  \mathtt{Sup}_V(\overline{j}) = \bigvee_w \overline{j}(w)\otimes w, \mbox{ where } \overline{j}(w)=  \bigvee_v [w,v] \otimes j(v).$$
In order to show that $[V] \simeq V$, we need to prove that the two structures are the same, that is to say
	$$\bigvee_w j(w)\otimes w = \bigvee_w (\bigvee_v [w,v] \otimes j(v))\otimes w.$$
	But since $[-,v] = \Yo{V}(v)$, we have $\mathtt{Sup}_V(\Yo{V}(v)) = \bigvee_v [w,v] \otimes v = v$. Thus the result follows.
 \end{re}
 \begin{re}
Remember that the (left) unitor in the monoidal category $\Alg{V}$, at an object $(X,\alpha)$, is the corresponding morphism to the bimorphism
$$\Pow{V}(1) \times X \rightarrow X, \quad (v,x) \mapsto \alpha(v \boxtimes e_X(x)) .$$
While the (left) unitor in the monoidal category $\Mod$, at an object $(X, \rho)$, is the corresponding morphism to the bimorphism
$$V \times X \rightarrow X, \quad (v,x) \mapsto \rho(v,x).$$
The equivalence $[-]$ sends $(X,\rho)$ to the cocomplete $V$-category $(X,a)$ that has $\rho$ as copower. The category $(X,a)$ is then sent to the $\Pow{V}$-algebra $(X,\alpha')$, where $\alpha' = \Sup_X(- \dist a)$. In this way, we get that the unitor
$$V \times X \rightarrow X, \quad (v,x) \mapsto \rho(v,x),$$
becomes the copower of $(X,a)$ which is then sent to the map
$$\Pow{V}(1) \times X \rightarrow X, \quad (v,x) \mapsto v \odot_{\rho} x.$$
In order to conclude, we notice that
\begin{alignat*}{2}
  \Sup_X(v \boxtimes e_X(x) \dist a) &= \Sup_X(v \boxtimes \Yo{X}(x)), \\
  & = v \odot_{\rho} x.
\end{alignat*}
 \end{re}
\section{First Interlude: Algebras and Modules}
In this section we study further the category $\Mon{\Mod}{V}$. In commutative algebra it is well known that monoids in the category of modules over a commutative ring $R$ are (associative and unital) $R$-algebras. In our case the quantale $V$ plays the role of the base ring $R$, thus one might expect that a similar result holds also for $\Mon{\Mod}{V}$. The answer is positive but it requires us to restrict our attention to a particular subcategory of $V \downarrow \mathtt{Quant}$.
\begin{defin}\label{Spades}
	We define $\Quant{Quant}$ to be the full subcategory of the coslice category $V \downarrow \mathtt{Quant}$ whose objects are morphisms of quantales $f : V \rightarrow Q$ such that, for all $v \in V, $ $u \in Q$, $f(v) \mult_Q u = u \mult_Q f(v)$.
\end{defin}
\begin{prop}
	\label{MonModQuant}
	There is an equivalence of categories: $\Mon{\Mod}{V} \simeq \Quant{Quant}.$
\end{prop}
\begin{proof}
  Since the monoidal structure on $\Mod$ is the one induced by a strong commutative monad, as we explained in Remark \ref{StrongMod}, by Theorem \ref{SCMon}, it follows that the functor
  $$ V \Tensor{2} = : \mathtt{Sup} \rightarrow \Mod, \qquad X \mapsto V\Tensor{2} X,$$
  is strong monoidal. This implies, by doctrinal adjunction (see \cite{10.1007/BFb0063105}), that the forgetful functor
  $$U : \Mod \rightarrow \mathtt{Sup} $$
  is lax monoidal. Let $X$ and $Y$ be $V$-modules. Then the laxator
  $$ \pi : X \Tensor{2}Y \rightarrow X \Tensor{V} Y,$$
  is the universal bimorphism that ``defines'' $\Tensor{V}$, while
  $$2 \rightarrow V$$
  is the canonical inclusion. Moreover, since $U$ is lax monoidal, it sends monoids in $\Mod$ to monoids in $\mathtt{Sup}$.\par \medskip

  Let $e : V \rightarrow X$, $ m :X \Tensor{V} X \rightarrow X$, $\alpha : V\Tensor{2}X \rightarrow X$ be an object of $ \Mon{\Mod}{V}$. Then $X = (X, m \cdot \pi, e \cdot k)$ is a monoid in $\mathtt{Sup}$, thus a quantale.
  In order to have an object in $\Quant{Quant}$, we have to show that $e : V \rightarrow X$ is a quantale homomorphism and that it satisfies the condition $e(v) \mult_X x = x \mult_X e(v)$ (where $\mult_X$ is a shortcut for $m \cdot \pi$, the multiplication of $X$ seen as a quantale).\\
	Let's start by proving that $e: V \rightarrow X$ is a quantale homomorphism. We notice that the following diagram
\begin{center}
	\begin{tikzcd}
	V \Tensor{\mathbf{2}} V \arrow[r, "e \Tensor{\mathbf{2}} e"] \arrow[d, "\otimes", swap]
	& X \Tensor{\mathbf{2}} X \arrow[d, "m \cdot \pi"] \\
	V \arrow[r, "e"]
	& X \end{tikzcd}
\end{center}
	can be decomposed as follows:
\begin{center}
	\begin{tikzcd}
	V \Tensor{\mathbf{2}} V \arrow[r, "\Id \Tensor{\mathbf{2}} e"] \arrow[d]
	& V \Tensor{\mathbf{2}} X \arrow[d, "\alpha", swap] \arrow[r, "e \Tensor{\mathbf{2}} \Id"] & X \Tensor{\mathbf{2}} X \arrow[d, "m \cdot \pi"] \\
	V \arrow[r, "e"] & X \arrow[r, equal]  & X.
	\end{tikzcd}
\end{center}
	The left diagram commutes because $e$ is an homomorphism of modules, while we can further decompose the right diagram as
\begin{center}
	\begin{tikzcd}
	V \Tensor{\mathbf{2}} X \arrow[r, "e \Tensor{\mathbf{2}} \Id"] \arrow[d, "\pi", swap] & X \Tensor{\mathbf{2}} X \arrow[d,  "\pi"] \\
	V \Tensor{V} X  \arrow[r, "e \Tensor{V} \Id"] \arrow[d, "\overline{l}_X", swap]  & X \Tensor{V} X \arrow[d, "m"] \\
	X \arrow[r, equal] & X,
	\end{tikzcd}
\end{center}
	where the bottom part commutes since $e$ is the unit of a monoid in $\Mod$.\\
	The condition $e(v) \mult_X x = x \mult_X e(v)$ follows by a direct contemplation of the following diagram,
\begin{center}
	\begin{tikzcd}
	V \Tensor{\mathbf{2}} X \arrow[r, "e \Tensor{\mathbf{2}} \Id"] \arrow[d, "\pi", swap]
	& X \Tensor{\mathbf{2}} X \arrow[d, "\pi"]  & X \Tensor{\mathbf{2}} V \arrow[d, " \pi", swap]  \arrow[l, "\Id \Tensor{\mathbf{2}} e", swap]\\
	V\Tensor{V} X \arrow[r, "e \Tensor{V} \Id"] \arrow[dr,bend right ] & X \arrow[d, "m", swap]  & X \Tensor{V} V. \arrow[l, "\Id \Tensor{V} e", swap]  \arrow[dl,bend left]\\
	& X &
	\end{tikzcd}
\end{center}
 Thus we have a functor
	$$ F : \Mon{\Mod}{V} \rightarrow \Quant{Quant}.$$
	This functor is clearly faithful, and with a little effort it is possible to show that $F$ is also full. Let $${f : (X, m \cdot \pi, e_X \cdot k) \rightarrow (Y, n \cdot \pi, e_X \cdot k)}$$ be a morphism of quantales where $(X, \rho, m, e_X)$, $(Y,\theta, n, e_Y)$ are monoids in $\Mod$. 
It is possible to show that $f$ is equivariant, that is to say,
$$ f \cdot \rho = \theta \cdot (\Id \Tensor{\mathbf{2}} f).$$
Consider the following diagram
\begin{center}
	\begin{tikzcd}
	X \Tensor{\mathbf{2}} X \arrow[r, "f \Tensor{\mathbf{2}} f"] \arrow[d, "\pi", swap] & Y \Tensor{\mathbf{2}} Y \arrow[d,  "\pi"] \\
	X \Tensor{V} X  \arrow[r, "f \Tensor{V} f"]  \arrow[d, "m", swap]  & Y \Tensor{V} Y \arrow[d, "n"] \\
  X \arrow[r, "f"] & Y.
	\end{tikzcd}
\end{center}
The outer and the upper diagrams commute, moreover, by the universal property of $\pi$, it follows that also the lower diagram commutes. Thus we can conclude that $f$ is a monoid homomorphism too.\par\medskip
Let us prove that $F$ is essentially surjective. Let $f : V \rightarrow Q, \ \ \mult_Q : Q \Tensor{\mathbf{2}} Q \rightarrow Q, \ \ e : \mathbf{2} \rightarrow Q$ be and object of $\Quant{Quant}$. We have that
\begin{center}
	\begin{tikzcd}
	V \Tensor{\mathbf{2}} Q \arrow[r, "f \Tensor{\mathbf{2}} \Id "] & Q \Tensor{\mathbf{2}} Q  \arrow[r, "\mult_Q"] & Q
	\end{tikzcd}
\end{center}
	defines an action, call it $\rho$. Indeed we have
\begin{center}
	\begin{tikzcd}
	\mathbf{2}  \Tensor{\mathbf{2}} Q \arrow[r, "k \Tensor{\mathbf{2}} \Id "] \arrow[dr, "e \Tensor{\mathbf{2}} \Id ", swap] \arrow[ddr, "l_Q", bend right, swap] & V \Tensor{\mathbf{2}} Q \arrow[d, "f \Tensor{\mathbf{2}} \Id "] \\
	& Q \Tensor{\mathbf{2}} Q \arrow[d, "\mult_Q"] \\
	& Q
	\end{tikzcd}
\end{center}
	where the inner triangle is the tensor with $- \Tensor{\mathbf{2}} Q$ of
\begin{center}
	\begin{tikzcd}
	\mathbf{2}  \arrow[r, "k"] \arrow[dr,"e", swap] & V \arrow[d, "f"] \\
	& Q
	\end{tikzcd}
\end{center}
	and the outer one is the unital condition for the multiplication of $Q$.\\
	The associativity condition follows from the fact that $f$ is a quantale homomorphism and from the associativity of $\mult_Q$, as depicted in the following commutative diagram
\begin{center}
	\begin{tikzcd}[column sep=huge]
	V \Tensor{\mathbf{2}} V \Tensor{\mathbf{2}} Q \arrow[r, "\Id  \Tensor{\mathbf{2}} (f \Tensor{\mathbf{2}} \Id ) "] \arrow[d, equal] & V \Tensor{\mathbf{2}} V \Tensor{\mathbf{2}} Q  \arrow[r, "\Id  \Tensor{\mathbf{2}} \mult_Q "] \arrow[d, "f \Tensor{\mathbf{2}} \Id ",swap] & V \Tensor{\mathbf{2}} Q \arrow[d, "f \Tensor{\mathbf{2}} \Id "] \\
	V \Tensor{\mathbf{2}} V \Tensor{\mathbf{2}} Q \arrow[r, "f \Tensor{\mathbf{2}} (f \Tensor{\mathbf{2}} \Id ) "] \arrow[d, "\otimes \Tensor{\mathbf{2}} \Id ",swap] & Q \Tensor{\mathbf{2}} Q \Tensor{\mathbf{2}} Q  \arrow[r, "\Id  \Tensor{\mathbf{2}} \mult_Q "] \arrow[d, "\mult_Q",swap] & Q \Tensor{\mathbf{2}} Q\arrow[d, "\mult_Q"] \\
	V \Tensor{\mathbf{2}} V  \arrow[r, "f \Tensor{\mathbf{2}} \Id  "] & Q \Tensor{\mathbf{2}} Q \arrow[r, "\mult_Q "] & Q.
	\end{tikzcd}
\end{center}
	We can prove that $\mult_Q : Q \Tensor{\mathbf{2}} Q \rightarrow Q$ coequalizes the fork that defines $Q \Tensor{V} Q$, hence that there is a unique $\overline{\mult}_Q : Q \Tensor{V} Q \rightarrow Q$ that makes the following diagram commute
\begin{center}
	\begin{tikzcd}
	Q \Tensor{\mathbf{2}} Q \arrow[d,"\pi", swap] \arrow[r, "\mult_Q"] & Q. \\
	Q \Tensor{V} Q \arrow[ur, "\overline{\mult}_Q", bend right]
	\end{tikzcd}
\end{center}
Let us prove this statement. In the fork that defines $Q \Tensor{V} Q$ the two arrows are defined as follows
\begin{center}
	\begin{tikzcd}
	V \Tensor{\mathbf{2}} Q \Tensor{\mathbf{2}} Q \arrow[r, "\rho \Tensor{\mathbf{2}} \Id "] & V \Tensor{\mathbf{2}} Q
	\end{tikzcd}
\end{center}
	and
\begin{center}
	\begin{tikzcd}
	V \Tensor{\mathbf{2}} Q \Tensor{\mathbf{2}} Q \arrow[r, "\tau \Tensor{\mathbf{2}} \Id "] & Q \Tensor{\mathbf{2}} V \Tensor{\mathbf{2}} Q \arrow[r, "\Id  \Tensor{\mathbf{2}} \rho"] & Q \Tensor{\mathbf{2}} Q.
	\end{tikzcd}
\end{center}
	But, since $\rho =  \mult_Q \cdot (f \Tensor{\mathbf{2}} \Id )$, and since the condition $f(v) \mult_Q u = u \mult_Q f(v)$ means
\begin{center}
	\begin{tikzcd}
	Q \Tensor{\mathbf{2}} Q  \arrow[r, "\mult_Q"] & Q & Q \Tensor{\mathbf{2}} Q \arrow[l, "\mult_Q", swap]  \\
	V \Tensor{\mathbf{2}} Q \arrow[u, "f \Tensor{\mathbf{2}} \Id "]  \arrow[rr, "\tau_{V,Q}"] & &    Q \Tensor{\mathbf{2}} V \arrow[u, "\Id  \Tensor{\mathbf{2}} f", swap]
	\end{tikzcd}
\end{center}
	by using $(\Id  \Tensor{\mathbf{2}} f) \Tensor{\mathbf{2}} \Id  = \Id  \Tensor{\mathbf{2}} (f \Tensor{\mathbf{2}} \Id  )$ and $\mult_Q \cdot (\Id  \Tensor{\mathbf{2}} \mult_Q) = \mult_Q \cdot (\mult_Q \Tensor{\mathbf{2}} \Id )$, we have
\begin{alignat*}{2}
	\mult_Q \cdot (\Id  \Tensor{\mathbf{2}} \rho) \cdot (\tau \Tensor{\mathbf{2}} \Id ) &= \mult_Q \cdot (\Id  \Tensor{\mathbf{2}} \mult_Q) \cdot (\Id  \Tensor{\mathbf{2}} (f \Tensor{\mathbf{2}} \Id )) \cdot (\tau \Tensor{\mathbf{2}} \Id )\\
	&=\mult_Q \cdot (\mult_Q \Tensor{\mathbf{2}} \Id )\cdot ((\Id  \Tensor{\mathbf{2}} f) \Tensor{\mathbf{2}} \Id ) \cdot (\tau \Tensor{\mathbf{2}} \Id ) \\
	&=\mult_Q \cdot ((\mult_Q \cdot (\Id \Tensor{\mathbf{2}} f) \cdot \tau)\Tensor{\mathbf{2}} \Id ) \\
	&= \mult_Q \cdot ((\mult_Q \cdot (f \Tensor{\mathbf{2}} \Id ))\Tensor{\mathbf{2}} \Id ) \\
	&= \mult_Q \cdot (\rho \Tensor{\mathbf{2}} \Id ).
\end{alignat*}

	The associativity of $\overline{\mult}_Q$ follows from the associativity of $\mult_Q$, 
  while the unit condition follows by a direct inspection of the following diagram
\begin{center}
	\begin{tikzcd}
	V \Tensor{\mathbf{2}} Q \arrow[r, "f \Tensor{\mathbf{2}} \Id "] \arrow[d, "\pi", swap] & Q \Tensor{\mathbf{2}} Q \arrow[d, "\pi", swap] & Q \Tensor{\mathbf{2}} V \arrow[d, "\pi", swap] \arrow[l, "\Id  \Tensor{\mathbf{2}} f", swap] & Q \Tensor{\mathbf{2}} V \arrow[l, "\simeq", swap] \arrow[d, "\pi"] \\
	V \Tensor{V} Q \arrow[r, "f \Tensor{V} \Id "] \arrow[dr, "\overline{l}_Q", bend right, swap] & Q \Tensor{V} Q \arrow[d, "\overline{\mult}_Q"] & Q \Tensor{V} V  \arrow[l, "\Id  \Tensor{V} f", swap]\arrow[dl, "\overline{r}_Q", bend left] & Q \Tensor{V} V. \arrow[l, "\simeq", swap] \arrow[dll, "\overline{l}_Q", bend left] \\
	& Q
	\end{tikzcd}
\end{center}
	Indeed, from
	$$ \overline{\mult}_Q \cdot (f \Tensor{V} \Id ) \cdot \pi =  \overline{\mult}_Q\cdot (\Id \Tensor{V} f) \cdot \tau \cdot \pi,$$
	if we cancel $\pi$, and if we notice that $\overline{\mult}_Q \cdot (f \Tensor{V} \Id ) =: \rho = \overline{l}_Q,$ we have
	$$\overline{l}_Q = \overline{\mult}_Q \cdot (\Id  \Tensor{V} f) \cdot \tau.$$
	Then, since $\overline{r}_Q = \overline{l}_Q \cdot \tau^{-1}$, we have $\overline{\mult}_Q\cdot (\Id  \Tensor{V} f) = \overline{r}_Q.$ \par \medskip

	The last thing we need to show is that $f : V \rightarrow Q$ is equivariant. Since $f$ is a morphism of quantales, and since the action on $V$ is its multiplication, the result follows directly from the following commutative diagram
\begin{center}
	\begin{tikzcd}
	V \Tensor{\mathbf{2}} V \arrow[r, "\Id  \Tensor{\mathbf{2}} f"] \arrow[d, "\otimes", swap] & V \Tensor{\mathbf{2}} Q \arrow[r, "f \Tensor{\mathbf{2}} \Id "] \arrow[dr, "\rho"] & Q \Tensor{\mathbf{2}} Q \arrow[d, "\mult_Q"] \\
	V  \arrow[rr, "f"]& &  Q.
	\end{tikzcd}
\end{center}
This ends the proof of the proposition.\\
\end{proof}
\begin{re}
	Note that in the case in which $V = \mathbf{2}$, $\Quant{Quant} \simeq \mathtt{Quant}$. Every morphism of quantales $f : \mathbf{2} \rightarrow Q$ satisfies $f(v) \cdot u = u \cdot f(v)$ and:
	$$ \mathbf{2} \downarrow \mathtt{Quant} \simeq \mathtt{Quant}.$$
\end{re}
\section{Second Interlude: Injectives and Monoids}\label{InjMon}
The enriched powerset monad $\Pow{V}$ is a strong commutative monad. This means (see Proposition \ref{StrongPow}) that the free functor
$$ \Pow{V} : \Sets \rightarrow \Alg{V}$$
is strong monoidal. Thus $\Pow{V}$ extends to a functor, which we denote again as $\Pow{V}$, between the corresponding categories of monoids:
$$\Pow{V} :  \mathtt{Mon} \rightarrow \Mon{\Alg{V}}{\Pow{V}}.$$
\begin{re}
	Let $(M, \cdot, e)$ be a monoid. Then $\Pow{V}(M)$ is the free $\Pow{V}$-algebra whose monoid structure is the following. The unit of $\Pow{V}(M)$ is
	\[ V \simeq \Pow{V}(1) \xrightarrow{\Pow{V}(e)} \Pow{V}(M). \]
	The multiplication of $\Pow{V}(M)$ is the unique $\Pow{V}$-morphism that corresponds to the bimorphism
	\[ \Pow{V}(M) \times \Pow{V}(M) \rightarrow \Pow{V}(M \times M) \xrightarrow{\Pow{V}(\cdot)} \Pow{V}(M), \]
	where
	\[ \Pow{V}(M) \times \Pow{V}(M) \rightarrow \Pow{V}(M \times M), \quad (\psi, \phi) \mapsto  \psi \boxtimes \phi : (m,n) \mapsto \psi(m)\otimes \phi(n).\]
\end{re}
Since $\Pow{V}$ is left adjoint to the forgetful functor
$$\Alg{V} \rightarrow \Sets,$$
by using Lemma \ref{MonoAdj}, we can conclude that the diagram
\begin{center}
	\begin{tikzcd}
		 \Mon{\Alg{V}}{\Pow{V}} \ar[r] \ar[d] & \mathtt{Mon}  \ar[d]\\
		 \Alg{V}\ar[r,] & \Sets
	\end{tikzcd}
\end{center}
commutes, where all functors forget part of the relevant structure.
\begin{lem} (\cite{Porst2008})\label{MonoAdj}
	Let $F: C \rightarrow C'$ be a monoidal functor between monoidal categories $C,C'$. If $F$ has a right adjoint $G$ with counit $\epsilon : FG \Rightarrow \Id _{C'}$. Then $\overline{F} : \mathtt{Mon}(C) \rightarrow \mathtt{Mon}(C')$ has a right adjoint $\overline{G}$ with counit $\overline{\epsilon}$ such that:
	\begin{itemize}
		\item The diagram

		\begin{tikzcd}
			\mathtt{Mon}(C') \ar[r,"\overline{G}"] \ar[d, "U'", swap] & \mathtt{Mon}(C)  \ar[d,"U"]\\
			 C' \ar[r, "G"] & C
		\end{tikzcd}

		commutes.
		\item $U'(\overline{\epsilon}_{(C,m,e)}) = \epsilon_C,$ for all $(C,m,e) \in \mathtt{Mon}(C')$.
	\end{itemize}
\end{lem}
The list functor $ L :\Sets \rightarrow \mathtt{Mon}$ is the left adjoint to the forgetful functor $\mathtt{Mon} \rightarrow \Sets$. The functor ${\Pow{V}:  \mathtt{Mon} \rightarrow \Mon{\Alg{V}}{\Pow{V}}}$ has as right adjoint the "lift" of the right adjoint of $\Pow{V}$ (which is the forgetful functor $\Alg{V} \rightarrow \Sets$). If we combine everything together, we discover that
$$ W : \Mon{\Alg{V}}{\Pow{V}} \rightarrow \Sets,$$
which is the functor obtained by composing the right adjoint to $\Pow{V}$ with the forgetful functor $\mathtt{Mon} \rightarrow \Sets,$ it is the right adjoint of
$$\Pow{V} L : \Sets \rightarrow \Mon{\Alg{V}}{\Pow{V}}.$$
We want to show that $ W : \Mon{\Alg{V}}{\Pow{V}} \rightarrow \Sets$ is monadic.
\begin{defin}
	$\mathtt{Alg}(T_{+}) $ is the category of algebras for the endofunctor
	$$ T_{+} : \Alg{V} \rightarrow \Alg{V}, \ \ X \mapsto (X \Tensor{\Pow{V}} X) \amalg V.$$
	That is to say the the category whose objects are of the form $m : (X \Tensor{\Pow{V}} X) \amalg V \rightarrow X$ (here $m$ is an arrow in $\Alg{V}$) and whose arrows $ f : (X, m) \rightarrow (Y, n)$ are those in $\Alg{V}$ that make the following diagram commute
	\begin{center}
		\begin{tikzcd}[row sep=large, column sep=huge]
		(X \Tensor{\Pow{V}} X) \amalg V  \arrow[d, "m", swap] \arrow[r, "(f \Tensor{\Pow{V}} f) \amalg \Id"] & (Y \Tensor{\Pow{V}} Y) \amalg V \arrow[d, "n"] \\
		X \arrow[r, "f"] & Y.
		\end{tikzcd}
	\end{center}
\end{defin}
\begin{re}
	Because $W$ just forgets the structure, we can immediately define
	$$\overline{W} : \mathtt{Alg}(T_{+}) \rightarrow \Sets,$$
	again as the forgetful functor; it is clear that its restriction to $\Mon{\Alg{V}}{\Pow{V}}$ is $W$.
	\end{re}
\begin{lem}
	\label{Epi}
	$\Mon{\Alg{V}}{\Pow{V}}$ is closed in $\mathtt{Alg}(T_{+}) $ with respect to epimorphisms.
\end{lem}
\begin{proof}
	(Based on \cite{Porst2008}, Proposition 2.6) \\
Let $(N, n, e_n)$ be an object of $\Mon{\Alg{V}}{\Pow{V}}$ and let $ d : (N, n, e_n) \rightarrow (M, m, e_m)$ be an epimorphism in $\mathtt{Alg}(T_{+}) $.\\
	From diagram chasing over
	\begin{center}
		\begin{tikzcd}[row sep=10, column sep=3]
		& N \Tensor{\Pow{V}} N \Tensor{\Pow{V}} N  \arrow[dl, "d \Tensor{\Pow{V}} d \Tensor{\Pow{V}} d", swap] \arrow[rr, "n \Tensor{\Pow{V}} \Id "] \arrow[dd,  "\Id  \Tensor{\Pow{V}} n ", near end, swap] & & N \Tensor{\Pow{V}} N \arrow[dl, "d \Tensor{\Pow{V}} d"]  \arrow[dd, "n"]  \\ M \Tensor{\Pow{V}} M  \Tensor{\Pow{V}} M \arrow[rr, crossing over, "m \Tensor{\Pow{V}} \Id ", near end] \arrow[dd, "\Id  \Tensor{\Pow{V}} m", swap] & & M \Tensor{\Pow{V}} M\\
		& N \Tensor{\Pow{V}} N  \arrow[dl, "d \Tensor{\Pow{V}} d", swap]  \arrow[rr, "n", near end]  & & N  \arrow[dl, "d"] \\
		M \Tensor{\Pow{V}} M \arrow[rr, "m"] & & M  \arrow[from=uu, crossing over, "m"]\\
		\end{tikzcd}
	\end{center}
we get
$$ m \cdot (m \Tensor{\Pow{V}} \Id ) \cdot (d \Tensor{\Pow{V}} d \Tensor{\Pow{V}} d) = m \cdot (\Id  \Tensor{\Pow{V}} m ) \cdot (d \Tensor{\Pow{V}} d \Tensor{\Pow{V}} d),$$
which implies, since $d \Tensor{\Pow{V}} d \Tensor{\Pow{V}} d$ is an epimorphism (being $\Tensor{\Pow{V}}$ a closed structure), that
$$m \cdot (m \Tensor{\Pow{V}} \Id ) = m \cdot (\Id  \Tensor{\Pow{V}} m ).$$
In a similar way one proves the corresponding equation for the unit from which it follows that $(M, m, e_m)$ is an object of $\Mon{\Alg{V}}{\Pow{V}}$.\\
\end{proof}
\begin{lem}
	\label{Coeq}
	The functor $\overline{W} : \mathtt{Alg}(T_{+}) \rightarrow \Sets$ creates coequalizers of $\overline{W}$-split pairs.
\end{lem}
\begin{proof}
		Let $R \rightrightarrows X$ be a $\overline{W}$-split pair. Let
		\[
		\begin{tikzcd}
		R \arrow[r, shift left]
		\arrow[r, shift right] & X \arrow[l,  "t", bend right, swap] \arrow[r, "\pi"] &  Q, \arrow[l,  "s", bend right, swap]
		\end{tikzcd}
		\]
	  be its (splitting) coequalizer in $\Sets.$\\
		Since the monoidal structure in $\Alg{V}$ classifies bimorphisms, we have that, associated to the "monoid structure without equations" of $R$ and $X$: $\tilde{r} :R \Tensor{\Pow{V}} R \rightarrow R$ and $\tilde{m} :X \Tensor{\Pow{V}} X \rightarrow X,$ there exist two unique bimorphisms $r : R \times R \rightarrow R$ and $m : X\times X \rightarrow X$.\\
		Because $\pi$ is an epimorphism and $\times$ is closed, it follows that $\pi \times \pi$ is an epimorphism too, hence
		\[
		\begin{tikzcd}
		R \times R  \arrow[r, shift left]
		\arrow[r, shift right] & X \times X  \arrow[l, "t \times t",bend right, swap ] \arrow[r, "\pi \times \pi"] &  Q \times Q \arrow[l, "s \times s", bend right, swap]
		\end{tikzcd}
		\]
		is again a split coequalizer in $\Sets$.
		Moreover, since we obtained $\overline{W}$ first by forgetting the "free monoid structure" and then by forgetting the $\Pow{V}$-structure, and since $\Alg{V}$ is monadic over $\Sets$, it follows that there exists a unique $\Pow{V}$-structure on $Q$ such that $\pi$ becomes a $\Pow{V}$-algebra morphism.\par \medskip
		Since $R \rightrightarrows X$ are both $T_{+}$-morphisms, by using the split and the universal property of coequalizers we get a unique function $ n : Q \times Q \rightarrow Q$, as displayed in the following diagram
		\begin{equation}\label{Comm}
		\begin{tikzcd}
		R \times R  \arrow[r, shift left]
		\arrow[r, shift right] \arrow[d]
		& X \times X\arrow[r] \arrow[d ]  & Q \times Q \arrow[d, "n", dashed] \\
		R \arrow[r, shift left]
		\arrow[r, shift right]
		& X \arrow[r]  & Q.
		\end{tikzcd}
	\end{equation}
 Since $\pi \cdot m \cdot (\pi \times \pi)$ and $\pi \cdot m$ are both bimorphisms in $\Alg{V}$, it follows that $n$ is a bimorphism too. Indeed, fix $q \in Q$, then---in $\Sets$---we have that the diagram
		\begin{center}
		\begin{tikzcd}
			X \arrow[r, "\pi"] \arrow[d, "\langle x \mbox{,} \Id\rangle ", swap] & Q \arrow[d, "\langle q \mbox{,} \Id\rangle"] \\
			X \times X \arrow[r] \arrow[d, "m", swap] & Q \times Q \arrow[d,"n"] \\
				X \arrow[r, "\pi"] & Q
			\end{tikzcd}
		\end{center}
	commutes, where $\pi(x) = q$ and $\langle q, \mbox{Id} \rangle (w) =  (q,w).$ Because $\pi$ is and epimorphism and
	$$ \pi \cdot m \cdot \langle x, \mbox{Id} \rangle : X \rightarrow Q$$
	is a $\Pow{V}$-algebra morphism, it follows that $n \cdot \langle q, \mbox{Id} \rangle$ is a $\Pow{V}$-algebra morphism too. We can do the same for $\langle  \mbox{Id}, q \rangle$, thus showing that $n$ is a bimorphism. This shows that there exists a unique arrow
	$ \tilde{n} : Q \Tensor{\Pow{V}} Q \rightarrow Q,$ in $\Alg{\Pow{V}}$ that makes the diagram
	\[
	\begin{tikzcd}
		Q \times Q \ar[r, "n"]  \ar[d]  & Q \\
		Q \Tensor{\Pow{V}} Q \ar[ur,"\tilde{n}", swap]
	\end{tikzcd}
	\]
	commute.\\
	In order to have an object of $\mathtt{Alg}(T_{+})$, we also need a $\Pow{V}$-algebra morphism
	$$ V \rightarrow Q.$$
	Since $X$ is in $\mathtt{Alg}(T_{+})$, we have $I_X :V \rightarrow X$; by taking the composite $I_Q :=  \pi \cdot I_X :  V \rightarrow X \rightarrow Q$ we get our desired arrow. In order to show that $\pi : X \rightarrow Q$ is in $\mathtt{Alg}(T_{+})$, we appeal to the following diagram
	\[
	\begin{tikzcd}
	(R \Tensor{\Pow{V}}  R) \amalg V  \arrow[r, shift left]
	\arrow[r, shift right] \arrow[d, "\tilde{r} \amalg I_R", swap]
	& (X \Tensor{\Pow{V}} X) \amalg V \arrow[r] \arrow[d, " \tilde{m} \amalg I_X", swap]  & (Q \Tensor{\Pow{V}} Q) \amalg V  \arrow[d, "\tilde{n} \amalg I_Q"] \\
	R \arrow[r, shift left]
	\arrow[r, shift right]
	& X \arrow[r]  & Q
	\end{tikzcd}
	\]
	whose commutativity follows from the universal property of bimorphisms and from Diagram \ref{Comm}. \par \medskip
  In order to conclude our proof we are left to show that $Q$ is the coequalizer of $R\rightrightarrows X$ in $\mathtt{Alg}(T_{+})$.\\
	Since we have already noticed that $Q$ is the coequalizer of $R\rightrightarrows X$ in $\Alg{V}$, for every (appropriate) arrow $g : X \rightarrow E$ in $\mathtt{Alg}(T_{+})$, we get a unique $\Pow{V}$-algebra morphism such that the following diagram
	\begin{center}
		\begin{tikzcd}
		R \arrow[r, shift left]
		\arrow[r, shift right]  & X \arrow[r, "\pi"] \arrow[d, "g", swap] & Q \arrow[ld, dashed, "f"] \\
		                                  & E
		\end{tikzcd}
	\end{center}
commutes.
In this way we get a unique morphism $f \Tensor{\Pow{V}} f : Q \Tensor{\Pow{V}} Q \rightarrow E \Tensor{\Pow{V}} E$. What we want to show is that $f : E \rightarrow Q$ is a morphism in $\mathtt{Alg}(T_{+})$, thus that the following diagram commutes, where the vertical arrows are the multiplication on $Q$ and $E$ respectively,
\[
\begin{tikzcd}
 Q \Tensor{\Pow{V}} Q \arrow[d, "\tilde{n}", swap] \arrow[r, "f \Tensor{\Pow{V}} f "] & E \Tensor{\Pow{V}} E \arrow[d, "\tilde{e}"] \\
 Q \arrow[r,"f"] & E
\end{tikzcd}
\]
which would then imply the commutativity of
\[
\begin{tikzcd}[row sep=large, column sep=huge]
(Q \Tensor{\Pow{V}} Q) \amalg V  \arrow[d] \arrow[r,"(f \Tensor{\Pow{V}} f) \amalg \Id "] & (E \Tensor{\Pow{V}} E) \amalg V \arrow[d] \\
Q \arrow[r, "f \amalg \Id "] & E.
\end{tikzcd}
\]
Since $R \times R \rightrightarrows X \times X \rightarrow Q \times Q$ is a coequalizer in $\Sets$, it follows that the following diagram commutes
\begin{equation}\label{Primo}
	\begin{tikzcd}
	Q \times Q \arrow[d, "n", swap] \arrow[r,"f \times f"] & E \times E \arrow[d, "e"] \\
	Q \arrow[r,"f"] & E
	\end{tikzcd}
\end{equation}
were $e : E \times E \rightarrow E$ is the unique bimorphism associates to $\tilde{e} :E \Tensor{\Pow{V}} E \rightarrow E $. Indeed, from the following commutative diagram
\begin{center}
	\begin{tikzcd}[row sep=10, column sep=3]
	& E \times E    \arrow[dd,  "e", near end, swap] & &   \\ X \times X \arrow[rr, crossing over, "\pi \times \pi", near end] \arrow[dd, "m", swap] \arrow[ur, "g \times g"] & & Q \times Q \arrow[ul,"f \times f", swap]\\
	& E  \\
	X\arrow[rr, "\pi"]  \arrow[ur, "g"] & & Q  \arrow[from=uu, crossing over, "q"] \arrow[ul, "f", swap] \\
	\end{tikzcd}
\end{center}
it follows that
$$e \cdot (f \times f) \cdot (\pi \times \pi) = f \cdot q \cdot (\pi \times \pi),$$
which implies that
$$ e \cdot (f \times f) = f \cdot q,$$
since $\pi \times \pi$ is an epi.\\
This allows us to conclude. Indeed, from the commutativity of Diagram \ref{Primo}, by using the universal property of bimorphisms, we can deduce the commutativity of the following diagram
\[
\begin{tikzcd}
 Q \Tensor{\Pow{V}} Q \arrow[d, "\tilde{n}", swap] \arrow[r, "f \Tensor{\Pow{V}} f "] & E \Tensor{\Pow{V}} E \arrow[d, "\tilde{e}"] \\
 Q \arrow[r,"f"] & E
\end{tikzcd}
\]
as required. This ends the proof of the lemma.\\
\end{proof}
\begin{prop}
	\label{MonMon}
	The functor $ W : \Mon{\Alg{V}}{\Pow{V}} \rightarrow \Sets$ is monadic.
\end{prop}
\begin{proof}
	Let $R \rightrightarrows X$ be a $W$-split pair in $\Mon{\Alg{V}}{\Pow{V}}.$ By Lemma \ref{Coeq} we know that there exists the coequalizer of following diagram in $\mathtt{Alg}(T_{+})$
	\begin{center}
	\begin{tikzcd}
		R \arrow[r, shift left]
		\arrow[r, shift right] & X \arrow[r, "e"] &  Q.
	\end{tikzcd}
	\end{center}
Since $e$ is an epi in $\mathtt{Alg}(T_{+})$, by Lemma \ref{Epi}, it follows that $Q$ is an object of $\Mon{\Alg{V}}{\Pow{V}}.$ Moreover, since $\Mon{\Alg{V}}{\Pow{V}} \rightarrow \mathtt{Alg}(T_{+})$ is fully faithful
$$ Q = \mathtt{Coeq}(R \rightrightarrows X).$$
Now, from $\restr{\overline{W}}{\Mon{\Alg{V}}{\Pow{V}} } = W$, it also follows that $Q$ is preserved by $W$.\\
The fact that $W$ reflects isomorphisms is straightforward.\\
\end{proof}
The last proposition shows that objects of $\Mon{\Alg{V}}{\Pow{V}}$ are algebras for the monad induced by the adjunction $\Pow{V} L \dashv W$. 
In order to study better this monad, we use the following remark.
\begin{re}\label{Comp}
  Suppose that we have two adjunctions
  $$ F' \dashv G' : C \leftrightarrows D \mbox{ and } F \dashv G : D \leftrightarrows E,$$
  with units and counits given by
  $$\eta' : \Id \Rightarrow G'F',  \ \ \eta: \Id \Rightarrow GF,$$
  $$ \epsilon' : F'G' \Rightarrow \Id, \ \ \epsilon : FG \Rightarrow \Id.$$
  Then $ F F' \dashv G'G$, with unit and counit given by
	$$\overline{\eta} =  G' \cdot \eta_{F'} \cdot \eta',$$
	$$\overline{\epsilon} =  \epsilon \cdot F(\epsilon').$$
\end{re}
If we apply Remark \ref{Comp} to the adjuction
$$\Pow{V} \dashv U : \mathtt{Mon} \leftrightarrows \Mon{\Alg{V}}{\Pow{V}},$$
whose unit and counit, at a monoid $(M, \cdot, 1_M)$ and at an object $(Q,\alpha, \mult_Q, k_Q)$ of $\Mon{\Alg{V}}{\Pow{V}}$, are
$$
	\eta_M : M \xrightarrow{\ulcorner \Delta_M \urcorner} \Pow{V}(M),
$$
(where $\ulcorner \Delta_M \urcorner$ is the transpose of the diagonal $\Delta_M : M \times M \rightarrow V$) and
$$
	\epsilon_Q : \Pow{V}(Q) \xrightarrow{\alpha}  Q,
$$
and to the adjuction
$$L  \dashv U' : \Sets \leftrightarrows \mathtt{Mon},$$
whose unit and counit, at a set $X$ and at a monoid $(M,\cdot, 1_M)$, are
$$
	\eta'_X : X  \rightarrow L(X), \mbox{ } x \mapsto (x),
$$
$$
	\epsilon'_{M} : L(M) \xrightarrow{}  M, \mbox{ } (m_1,...,m_n) \mapsto m_1 \cdot ... \cdot m_n,
$$
by applying Remark \ref{Comp}, we get that the unit and the counit of the adjunction $\Pow{V} L \dashv W$ are
$$
	\overline{\eta}_X : X  \rightarrow \Pow{V}(L(X)), \mbox{ } x \mapsto \ulcorner \Delta_X \urcorner((x)),
$$
$$
	\overline{\epsilon}_Q : \Pow{V}(L(Q)) \xrightarrow{\Pow{V}(\epsilon'_Q)}  \Pow{V}(Q) \xrightarrow{\alpha}  Q.
$$
Hence the monad structure on $\Pow{V} L$ is defined as
$$
	\eta_X = \overline{\eta}_X : X \rightarrow  \Pow{V}(L(X)), \mbox{ } x \mapsto \ulcorner \Delta_X \urcorner((x)),
$$
$$
	\mu_X =  W\overline{\epsilon}_{\Pow{V} L} : (\Pow{V}L)(\Pow{V}L)(X)\rightarrow \Pow{V}L(X).
$$

Let us decompose a little bit more the multiplication. First of all, we notice that the $\Pow{V}$-structure $\Pow{V}L(X)$ possesses is the multiplication of the enriched powerset monad $\Pow{V}$ at $L(X)$
$$n_X : \Pow{V}\Pow{V}(LX) \rightarrow \Pow{V}(LX), \ \  n_X(\Phi)(\underline{x}) =  \bigvee_{\phi \in V^{L(X)}} \Phi(\phi) \otimes \phi(\underline{x}).$$
Notice that we can write $n_X(\Phi)(\underline{x})$ as the relational composite of $\Phi$ viewed as a $V$-relation
$$ \Phi: \Pow{V}L(X) \xslashedrightarrow{} 1,$$
with the $V$-relation
$$\mathtt{ev}_{LX} : LX \xslashedrightarrow{} \Pow{V}L(X), \ \ \mathtt{ev}_{LX}(\underline{x}, \phi) =  \phi(\underline{x}).$$
Interestingly enough, we can also write $\Pow{V}(\epsilon'_{\Pow{V} L})$ as the composition of $V$-relations. Indeed, for an element $\psi \in \Pow{V}(LX)$ seen as a $V$-relation
$$\psi : L(X) \xslashedrightarrow{} 1,$$
we have that $\Pow{V}(\epsilon'_{\Pow{V} L})(\psi) =  \epsilon^{'\circ}_{\Pow{V} L} \dist \psi$.\\
In this way we can write the multiplication $\mu_X$ as the composite of
$$LX \xslashedrightarrow{\mathtt{ev}_{LX}} \Pow{V} L X \xslashedrightarrow{\epsilon_{\Pow{V}L}^{' \circ}} L \Pow{V} LX \xslashedrightarrow{(-)} 1.$$

In Section \ref{CocoMonad} we stated that
$$\Inj{V} \simeq \Sets^{\Pow{L}},$$
where $\Pow{L}$ is the $\Sets$-monad we obtained by composing the presheaf monad $\mathbb{D}_L : \Multi{V} \rightarrow \Multi{V}$ with the "discrete" functor $d : \Sets \rightarrow \Multi{V}$. We showed that its unit and multiplication at $X$ are given by
$$ e_X : X \rightarrow \Pow{L}(X), \mbox{ } x \mapsto \Yo{X}(x) = e^{\circ}_X(-,x),$$
$$ n_X =  - \ldist \Yo{\circledast} : \Pow{L}\Pow{L}(X) \rightarrow \Pow{L}(X).$$
A brief calculation shows that
$$ \Pow{L}(X) =  \Dist{(L,V)}(d(X),E) = \Mat{V}(L(X),1) =  \Pow{V}(L(X)).$$
Indeed, since the $(L,V)$-structure on $d(X)$ is $e^{\circ}_X$, we have
$$ - \ldist e^{\circ}_X =   - \dist L(e^{\circ}_X) \cdot m^{\circ}_X = - \dist L(m _X \cdot e_X ) = - \dist \Id,$$
which shows that every relation $L(d(X)) \xslashedrightarrow{}  1$ is a distributor.\\
Notice that $e^{\circ}_X(-,x) = \ulcorner \Delta_X \urcorner((x))$ and also that the two $V$-relations $\mathtt{ev}_{LX}$ and $\Yo{\circledast}$, by the Yoneda lemma, are the same.
\begin{prop}\label{QuantInj}
There is an equivalence of categories: \[\Mon{\Alg{V}}{\Pow{V}} \simeq \Inj{V}.\]
\end{prop}
\begin{proof}
	If we prove that $\Pow{V}  L \simeq \Pow{L}$ as monads, i.e. in the sense of \cite{STREET1972149}, we are able to prove our proposition.\\
  As we noticed, the unit of $\Pow{V}  L $ and the unit of $\Pow{L}$ are the same. Hence, in order to conclude, we have to show that also the two multiplications are compatible.\\
  If we decompose the two multiplications, we have that pointwise they are defined as
	$$ LX \xslashedrightarrow{m_X^{\circ}} LLX \xslashedrightarrow{L\Yo{\circledast}} L\Pow{V}LX \xslashedrightarrow{(-)} 1 \quad \mbox{(multiplication of $\Pow{L}$)}$$
	and as
	$$LX \xslashedrightarrow{\mathtt{ev}_{LX}} \Pow{V} L X \xslashedrightarrow{\epsilon_{\Pow{V}L}^{' \circ}} L \Pow{V} LX \xslashedrightarrow{(-)} 1. \quad \mbox{(multiplication of $\Pow{V}L$)}$$
  Because in $\Mat{V} $ $\Yo{\circledast}$ is the same as $\mathtt{ev}_{LX}$, once we have shown that in $\Mat{V}$ the following diagram commutes
\begin{equation}\label{diagram2}
		\xymatrix{
			LX \ar[r]|-@{|}^{\Yo{\circledast}} \ar[d]|-@{|}_{m_X^{\circ}} & \Pow{V}LX \ar[d]|-@{|}^{\epsilon_{\Pow{V}L}^{' \circ}} \\
			LLX \ar[r]|-@{|}^{L\Yo{\circledast}}         & L\Pow{V}LX   }
\end{equation}
	we can conclude that the two monads are the same. \par \medskip
  Notice that the monoid structure on $\Pow{V}L(X)$ is defined as
	 \begin{alignat*}{2}
	\Pow{V}(LX)\times \Pow{V}(LX) &\longrightarrow \Pow{V}(LX \times LX)\longrightarrow \Pow{V}L(X)\\
	(\psi, \phi)&\longmapsto \psi \Tensor{} \phi \longmapsto m_2 \cdot (\psi \Tensor{} \phi )
	\end{alignat*}
	where
	\begin{alignat*}{2}
	m_2  : LX \times LX &\longrightarrow LX\\
	(\underline{x}, \underline{y}) &\longmapsto m_X(\underline{x}, \underline{y}).
	\end{alignat*}
	Hence it follows that $\epsilon'_{\Pow{V}L}$ is the composite
	$$L(\Pow{V}L(X)) \xslashedrightarrow{\amalg \otimes^n} \Pow{V}(LLX) \xrightarrow{\Pow{V}(m_X)} \Pow{V}L(X),$$
	where, for a list $ \underline{\phi} \in  L(\Pow{V}L(X))$, $(\amalg \Tensor{}^n)(\underline{\phi}) =  \phi_1 \Tensor{} ... \Tensor{} \phi_n$.\\
	In this way we can decompose Diagram \ref{diagram2} as follows
\[
		\xymatrix{
			LX \ar[r]|-@{|}^{\Yo{\circledast}} \ar[d]|-@{|}_{m_X^{\circ}} & \Pow{V}LX \ar[d]|-@{|}^{\Pow{V}(m_X)^{\circ}} \\
			LLX \ar[r]|-@{|}^{(\Yo{LX})_{\circledast}} \ar[dr]|-@{|}_{L\Yo{\circledast}} & \Pow{V}LLX \ar[d]|-@{|}^{(\amalg \Tensor{}^n)^{\circ}} \\
			& L\Pow{V}LX.   }
\]
We can easily prove that the two subdiagrams commute. Let $\underline{x} \in LX$ and $\phi \in \Pow{V}LX$, then we have
\begin{alignat*}{2}
\Pow{V}(m_X)^{\circ} \ldist \Yo{\circledast}(\underline{x}, \phi)  &= \Yo{\circledast}(\underline{x}, \Pow{V}(m_X)(\phi))  \\
&= \Pow{V}(m_X)(\phi)(\underline{x}) \\
&   \mbox{(from the Yoneda lemma)} \\
&= \bigvee_{ \{ \underline{\underline{x}}\in LLX \mbox{ | } m_X(\underline{\underline{x}}) = \underline{x}\}} \phi(\underline{\underline{x}}) \\
&  \mbox{(by definition of $\Pow{V}$)}  \\
&= \bigvee_{ \underline{\underline{x}}\in LLX} m_X^{\circ}(\underline{x}, \underline{\underline{x}}) \otimes \Yo{\circledast}(\underline{\underline{x}}, \phi) \\
& \mbox{(from the Yoneda lemma applied to $\phi(\underline{x}))$} \\
&= ((\Yo{LX})_{\circledast} \ldist m_X^{\circ}) (\underline{x}, \phi),
\end{alignat*}
which proves the commutativity of the upper square. If we fix $\underline{\underline{x}} \in LLX$ and $\underline{\phi} \in L\Pow{V}LX$, we have
\begin{eqnarray}
(\amalg \Tensor{}^n)^{\circ} \ldist (\Yo{LX})_{\circledast} (\underline{\underline{x}}, \underline{\phi}) &=& (\amalg \Tensor{}^n)(\underline{\phi})(\underline{\underline{x}} )  \nonumber \\
&=& \phi_1(\underline{x}_1) \otimes ... \otimes \phi_n(\underline{x}_n)   \nonumber \\
&=&L\Yo{\circledast} (\underline{\underline{x}}, \underline{\phi}), \nonumber
\end{eqnarray}
which proves the commutativity of the lower triangle and concludes the proof of the proposition.\\
\end{proof}
\begin{re}
 There is another interesting and conceptual way to prove that
 \[
 		\xymatrix{
 			LX \ar[r]|-@{|}^{\Yo{\circledast}} \ar[d]|-@{|}_{m_X^{\circ}} & \Pow{V}LX \ar[d]|-@{|}^{\Pow{V}(m_X)^{\circ}} \\
 			LLX \ar[r]|-@{|}^{(\Yo{LX})_{\circledast}} & \Pow{V}LLX }
 \]
commutes. Consider it as a diagram in $\Dist{V}$ with $LX$ seen as a discrete $V$-category and use the fact that $\Yo{\circledast}$ is the unit of a monad---hence a natural transformation in $\Dist{V}$. Notice that $\Pow{V}(m_X)^* = \Pow{V}(m_X)^{\circ}$ and, because $m_X \dashv m_X^*$ in $\Dist{V}$, we get that $\Pow{V}(m_X) \dashv \Pow{V}(m_X^*)$ and $\Pow{V}(m_X) \dashv \Pow{V}(m_X)^*$; by unicity of adjoints, it follows that $\Pow{V}(m_X^*) = \Pow{V}(m_X)^*$. In this way, from the commutativity of
\[
	 \xymatrix{
		 LLX  \ar[r]|-@{|}^{\Yo{\circledast}} \ar[d]|-@{|}_{m_X} & \Pow{V}LLX \ar[d]|-@{|}^{\Pow{V}(m_X)} \\
		 LX\ar[r]|-@{|}^{(\Yo{LX})_{\circledast}} & \Pow{V}LX }
\]
it follows the commutativity of the desired one, since
$$\Pow{V}(m_X^{\circ}) = \Pow{V}(m_X^*) = \Pow{V}(m_X)^* = \Pow{V}(m_X)^{\circ}.$$
\end{re}
\begin{cor}
	\label{cor}
	$\Inj{\mathbf{2}} \simeq \mathtt{Quant}.$
\end{cor}
\begin{proof}
	Since we have just proven that
	$$ \Inj{\mathbf{2}}\simeq \mathtt{Mon}(\Alg{\mathbf{2}}, \Tensor{\mathbf{2}},\mathbf{2}),$$
	from $\Alg{\mathbf{2}} \simeq \mathtt{Sup}$, and since---by definition---quantales are monoids in the category of suplattices, the result follows.\\
\end{proof}
\begin{re}
  In \cite{martinelli2020injective} the author gave another proof of the characterization of cocomplete multicategories exposed in Proposition \ref{QuantInj}. The main difference is the approach used; in \cite{martinelli2020injective}, the author obtained his result by using the machinery of $(L,V)$-colimits which are a generalization to the realm of $(L,V)$-categories of the notion of weighted colimits, while in Proposition \ref{QuantInj} we compared two monads. The advantage of the latter is that it gives a more manageable description of the category of algebras as a generalization to the enriched case of the notion of quantales. We must point out that the proof of Proposition \ref{QuantInj} came before \cite[Theorem~6.19]{martinelli2020injective} and it was the guiding principle that led to the proof of \cite[Theorem~6.19]{martinelli2020injective}.
\end{re}
\section{Conclusions}
We are now ready to conclude our tour de force and finally prove our desired result.
\begin{defin}
Let $\Mod(\mathtt{Quant})$ be the category whose objects are quantales $(Q,\mult, k_Q)$ equipped with an action $\rho : V \Tensor{2} Q  \rightarrow Q $ that is a monoid homomorphism and whose arrows are equivariant morphisms of quantales.
\end{defin}
\begin{re}\label{QEq}
  Notice that, to give an arrow $\rho : V \Tensor{2} Q \rightarrow Q$ in $\mathtt{Sup}$, is equivalent to give an arrow
  $$\rho' : V \Tensor{} Q\rightarrow Q$$
  in $\mathtt{Ord}$ that preserves suprema in each variable. Moreover, $\rho$ is an action iff $\rho'$ is an action. It is also true that $\rho$ is a monoid homomorphism iff $\rho'$ is a monoid homomorphism.
\end{re}
\begin{prop}\label{Last}
$\Mod(\mathtt{Quant}) \simeq \Quant{Quant}.$
\end{prop}
\begin{proof}
Let $ f : V \rightarrow (Q, \mult_Q,k_Q)$ be an object of $\Quant{Quant}$. Define the following function:
$$\rho'_f : V \Tensor{} Q \rightarrow Q, \quad (v,q) \mapsto f(v)\mult_Q q.$$
Because $f$ is a morphism of quantales and the multiplication of a quantale preserves suprema, it follows that $\rho'_f$ defines a unique arrow
$$\rho_f : V \Tensor{2} Q \rightarrow Q$$
in $\mathtt{Sup}$. It is straightforward to show that $\rho'_f$ is an action, hence $\rho_f$ is an action too.\\
By using the fact that $ f : V \rightarrow Q$ is an object of $\Quant{Quant}$, we can also prove that $\rho'_f$ is a monoid homomorphism. Indeed, let $v_1,v_2 \in V$ and $q_1,q_2 \in Q$, then we have
\begin{alignat*}{2}
\rho'_f((v_1,q_1)\mult_{V\Tensor{}Q} (v_2,q_2)) &= \rho'_f(v_1\otimes v_2, q_1 \mult_Q q_2) \\
&= f(v_1 \otimes v_2) \mult_Q q_1 \mult_Q q_2 \\
&= f(v_1) \mult_Q f(v_2) \mult_Q q_1 \mult_Q q_2 \\
&=f(v_1) \mult_Q q_1 \mult_Q f(v_2) \mult_Q q_2 \\
&=\rho'_f(v_1,x_1) \mult_{V\Tensor{}Q} \rho'_f(v_2,x_2).
\end{alignat*}
Thus $(Q,\rho)$ is an object of $\Mod(\mathtt{Quant})$. Let $h : Q \rightarrow W$, where $f: V \rightarrow Q$ and $g : V \rightarrow W$ are objects of $\Quant{Quant}$, be a morphism in $\Quant{Quant}$. It is straightforward to verify that $h : (Q, \rho_f) \rightarrow (W, \rho_g) $ is a morphism in $\Mod(\mathtt{Quant})$. Thus we have a functor
$$ F : \Quant{Quant} \rightarrow \Mod(\mathtt{Quant}).$$

Let $\rho : V \Tensor{2} (Q,\mult, k_Q) \rightarrow (Q,\mult, k_Q)$ be an object of $\Mod(\mathtt{Quant})$ and let $\rho' :  V \Tensor{} Q \rightarrow Q$ be as in Remark \ref{QEq}. Define the following morphism of quantales
$$f_{\rho'} : V \rightarrow Q, \quad v \mapsto \rho'(v,k_Q).$$
We have, for $q \in Q$, $v \in V$,
$$ f_{\rho'}(v)\mult_Q q = \rho'(v,k_Q) \mult_{V\Tensor{}Q}\rho'(k,q) = \rho'(v,q) = \rho'(k,q)\mult_{V\Tensor{}Q} \rho'(v,k_Q) = q \mult_Q f_{\rho'}(v).$$
If $h : (Q, \rho) \rightarrow (W,\theta)$ is an arrow in $\Mod(\mathtt{Quant})$, then $h\cdot f_{\rho'}= f_{\theta'}$. Thus we have a functor
$$G : \Mod(\mathtt{Quant}) \rightarrow  \Quant{Quant}.$$
Easy calculations show that $F$ and $G$ establish an equivalence between $\Mod(\mathtt{Quant})$ and $\Quant{Quant}$.\\
\end{proof}
\begin{re}
If $(X,\rho, \leq_X)$ is an object in $\Mod$, then the map $\rho(v, =) : X \rightarrow X$ defines a morphism in $\mathtt{Sup}$. Let $\rho : V \Tensor{} Q \rightarrow Q$ be an object of $\Mod(\mathtt{Quant})$. We might be tempted to see (or at least, the author was) if something similar holds. Unfortunately, $\rho(v,=) : Q \rightarrow Q$ does not define a morphism of quantales. Consider $q_1,q_2 \in Q$, then we would have
$$\rho(v,q_1 \mult_Q q_2)= \rho(v,q_1) \mult_Q \rho(v,q_2)$$
which in general is not true. In the previous proposition we showed that evey $\rho : V \Tensor{} Q \rightarrow Q$ is ``essentially'' of the form $f(-) \mult_Q =$, for an object $f :V \rightarrow Q$ of $\Quant{Quant}$. It is easy to see that
$$ \rho(v,q_1 \mult_Q q_2) := f(v)\mult_Q q_1 \mult_Q q_2, $$
in general is not equal to
$$\rho(v,q_1) \mult_Q \rho(v,q_2) := f(v)\mult_Q q_1 \mult f(v)\mult_Q q_2.$$
For example, one can take $Q = [0, \infty]^{\op}$, $V = [0, \infty]^{\op}$ and $f = \Id$.
\end{re}
The last proposition allows us to conclude:
\begin{teorema}
  $\Inj{V} \simeq  \Mod(\mathtt{Quant}).$
\end{teorema}
\begin{proof}
We have the following chain of equivalences
\begin{align*}
\Inj{V} &\simeq \Mon{\Alg{V}}{\Pow{V}} & \mbox{(by Proposition \ref{QuantInj})} \\
 &\simeq \Mon{\Mod}{V} & \mbox{(by Proposition \ref{MonBim})} \\
 & \simeq \Quant{Quant} &\mbox{(by Proposition \ref{MonModQuant})} \\
 & \simeq \Mod(\mathtt{Quant}). & \mbox{(by Proposition \ref{Last})}
\end{align*}
\end{proof}
\section*{Acknowledgements}

I am grateful to D. Hofmann for valuable discussions about the content of the paper and to I. Stubbe and A. Balan for the valuable suggestions they gave me during their visits to Aveiro. \\
The author acknowledges partial financial assistance by the ERDF – European Regional Development Fund through the Operational Programme for Competitiveness and Internationalisation - COMPETE 2020 Programme and by National Funds through the Portuguese funding agency, FCT - Fundação para a Ciência e a Tecnologia, within project POCI-01-0145-FEDER-030947, and project UID/MAT/04106/2019 (CIDMA). The author is also supported by FCT grant PD/BD/128187/2016.
\appendix
\section{Appendix: Strong Commutative Monads}\label{AppendixI}
The main focus of this appendix is to have a space in which we put some results about \textit{strong monads} that are used across the paper and which are neither suitable for being put into the narrative of the paper, neither for being simply cited. \par\medskip
Strong monads and strong commutative monads were introduced by Anders Kock (see \cite{AKSM}) as a way to study better certain types of monoidal monads.\par\medskip
In his article (\cite{JACOBS199473}), Bart Jacobs used the results obtained by Kock to study under which conditions the category of algebras of a monad defined on a monoidal category becomes itself monoidal. He also studied which conditions are needed in order to "reflect" other properties the base monoidal category might have.\par\medskip
In this appendix we present the main results of \cite{JACOBS199473} and we apply them to $\Coco{\Vcats}$, which, as we saw in Theorem \ref{SetMon}, is equivalent to the category of algebras for the $V$-powerset monad $(\Pow{V}, u, n)$.\par\medskip
The structure of this appendix is as follows:
\begin{itemize}
  \item In the first section, based on \cite{EGTG}, we study as a motivating/toy example, the monoidal structure $\mathtt{Sup}$ possesses;
  \item In the second section, following \cite{JACOBS199473}, we introduce strong commutative monads along the main results concerning them;
  \item In the last section we apply the results of the second section to the $V$-powerset monad $(\Pow{V}, u, n)$.
\end{itemize}
\section{Some Words about $\mathtt{Sup}$}
We recall that the powerset functor is defined by $\Pow{2}(X)=  2^X$, and by
$$ \Pow{2} : \Sets \rightarrow \Sets, \mbox{ } X \rightarrow Y \mapsto \Pow{2}f : \Pow{2}(X) \rightarrow \Pow{2}(Y), $$
where $\Pow{2}f ( A) = \{ f(x) \mbox{ such that } x\in A\}.$\par\medskip
It is well known that it is part of a monad $(\Pow{2}, u, n)$, called the powerset monad, where:
\begin{itemize}
  \item The unit at $X$ is given by $u_x : X \rightarrow \Pow{2}(X), \mbox{ } x \rightarrow \{x\};$
  \item The multiplication at $X$, $n_X : \Pow{2}(\Pow{2}(X)) \rightarrow \Pow{2}(X)$ is defined by $$n_X(\mathcal{A}) = \bigcup \{A \in \mathcal{A}\}.$$
\end{itemize}
It is straightforward to see that algebras for this monad are suplattices, where the free functor ${\Pow{2} :  \Sets \rightarrow \mathtt{Sup}}$ sends a set $X$ to $(\Pow{2}(X), \subseteq)$.\par\medskip
Let $X$ and $Y$ be suplattices. It is possible to form their tensor product $X \Tensor{2}Y$, defined as follows
$$X \Tensor{2}Y = \{C \in \Pow{}(X \times Y)\mbox{ | } \forall A \in \Pow{}(X),\forall B \in \Pow{}(Y), A\times B \subseteq C \iff (\bigvee A,\bigvee B) \in C \},$$
with the order structure induced by the one $P(X \times Y)$ has.\par\medskip
Moreover, $- \Tensor{2} =$ defines a symmetric closed monoidal structure on $\mathtt{Sup}$ with unit $\mathbf{2} \simeq \Pow{2}(1)$. With respect to this monoidal structure, since for $X, Y$, $$\Pow{2}(X)\Tensor{2}\Pow{}(Y) \simeq \Pow{2}(X \times Y),$$ the free functor $ \Pow{2} : \Sets \rightarrow \mathtt{Sup}$ becomes strong monoidal.\par\medskip
The tensor product just defined has another interesting property: it classifies bimorphisms. Let $X$, $Y$ and $Z$ be in $\mathtt{Sup}$, a bimorphism $f : X \boxtimes Y \rightarrow Z$ is a function such that, for all $x, y \in X,Y$,
$$f_x : Y \rightarrow Z, \mbox{ } y \mapsto f(x,y),$$
$$ f_y : X \rightarrow Z, \mbox{ } x \mapsto f(x,y),$$
are both suprema preserving maps.\\
This defines, for every $X,Y \in \mathtt{Sup}$, a functor
$$\Bim(X \boxtimes Y, = ) : \mathtt{Sup} \rightarrow \Sets,$$
where $\Bim(X\boxtimes Y, Z )$ denotes the sets of bimorphisms from $X\boxtimes Y$ to $Z$. The fact that $\Tensor{2}$ classifies bimorphisms means that there exists, for all $X,Y,Z \in \mathtt{Sup}$, a natural bijection
$$ \Bim( X\boxtimes Y, Z ) \simeq \mathtt{Sup}(X\Tensor{2} Y, Z)$$
which is realized by a universal bimorphism
$$ \pi : X\boxtimes Y \rightarrow X \Tensor{2} Y.$$
That is to say, for every bimorphism $f :  X \boxtimes Y \rightarrow Z$, there exists a unique suprema preserving map ${\overline{f} : X \Tensor{2} Y \rightarrow Z}$ that makes the following diagram
\[
\begin{tikzcd}
  X\boxtimes Y \ar[r, "f"] \ar[d, "\pi", swap] & Z \\
  X \Tensor{2} Y \ar[ur, "\overline{f}", dashed, swap]
\end{tikzcd}
\]
commute.\\
Before we conclude this section we must point out that that the tensor product $X \Tensor{2}Y$ can be also expressed as the coequalizer of the following parallel pair of arrows
$$ \Pow{2}(\Pow{2}(X) \times \Pow{2}(Y)) \xrightarrow{\Pow{2}t} \Pow{2}^2(X  \times Y )\xrightarrow{n} \Pow{2}(X  \times Y),$$
$$\Pow{2}(\Pow{2}(X) \times \Pow{2}(Y))\xrightarrow{\Pow{2}(\alpha \times \beta)} \Pow{2}(X  \times Y).$$
Here we have
$$ t : \Pow{2}(X) \times \Pow{2}(Y) \rightarrow \Pow{2}(X  \times Y ), \mbox{ } (A,B) \mapsto A \times B,$$
and
$$\alpha \times \beta : \Pow{2}(X  \times Y) \rightarrow X \times Y,  \mbox{ } (A \times B) \mapsto (\alpha(A), \beta(B)),$$
where $\alpha$ and $\beta$ are the algebras structures on $X$ and $Y$ respectively.\\
In the next section we will see how crucial the existence of an arrow like $t$ is in order to build the tensor product of algebras.
\section{Strong Commutative Monads}
\begin{defin}
Let $(C, \otimes, 1)$ be a monoidal category and $(T,e,m)$ be a monad with $T: C \rightarrow C$. The monad $(T,e,m)$ is called \textit{strong} if it is equipped with a natural transformation, called \textit{strenght}, with components
$$ \mathtt{st}_{X,Y}  : X \otimes TY \rightarrow T(X \otimes Y),$$
such that the following diagrams commute
\[
\begin{tikzcd}
&  TY &  \\
1 \otimes TY \ar[ur, "\simeq"] \ar[rr, "\mathtt{st}_{1,Y}"]&  & T(1 \otimes Y) \ar[ul, "\simeq", swap]
\end{tikzcd}
\]

\[
\begin{tikzcd}[row sep=large, column sep=large]
 (X \otimes Y) \otimes TZ \ar[rr, "\mathtt{st}_{X \otimes Y,Z}"] \ar[d, "\simeq"] &  & T((X \otimes Y) \otimes Z) \ar[d, "\simeq"] \\
 X \otimes (Y \otimes TZ)\ar[r, "\Id \otimes \mathtt{st}_{Y,Z}"] & X \otimes T(Y \otimes Z) \ar[r, "\mathtt{st}_{X,Y \otimes Z}"] & T(X \otimes (Y \otimes Z))
\end{tikzcd}
\]

\[
\begin{tikzcd}
&  X \otimes Y \ar[dr, "e_{X \otimes Y}"]  \ar[dl, "\Id \otimes e_Y",swap] &  \\
X \otimes TY  \ar[rr, "\mathtt{st}_{X,Y}"]&  & T(X \otimes Y)
\end{tikzcd}
\]

\[
\begin{tikzcd}[row sep=large, column sep=large]
  X \otimes T^2Y \ar[r, "\mathtt{st}_{X,TY}"] \ar[d, "\Id \otimes m_Y", swap] & T(X \otimes TY )  \ar[r, "T\mathtt{st}_{X,Y}"] & T^2(X \otimes Y) \ar[d, "m_{X \otimes Y}"], \\
  X \otimes TY \ar[rr, "\mathtt{st}_{X,Y}"] &  & T(X \otimes Y).
\end{tikzcd}
\]
\end{defin}

Suppose that $(C, \otimes, 1)$ is a symmetric monoidal category and let $(T,e,m)$ be a strong monad. Call $\gamma_{X,Y}: X \otimes Y \xrightarrow{\simeq } Y \otimes X$ the braiding. Then we define a \textit{co-strenght} as
$$ \mathtt{st}'_{X,Y} : TX \otimes Y \xrightarrow{\gamma_{TX,Y}} Y \otimes TX \xrightarrow{\mathtt{st}_{Y,X}} T(Y \otimes X) \xrightarrow{T\gamma_{X,Y}} T(X \otimes Y).$$
\begin{defin}
  Let $(C, \otimes, 1)$ be a symmetric monoidal category and let $(T,e,m)$ be a strong monad defined on it. The monad $(T,e,m)$ is called \textit{commutative} if the following diagram commutes
  \[
  \begin{tikzcd}[row sep=large, column sep=large]
    TX\otimes TY \ar[r, "\mathtt{st}_{TX,Y}"] \ar[d,"\mathtt{st}'_{X,TY}", swap] & T(TX \otimes Y) \ar[r, "T\mathtt{st}'_{X,Y}"] & T^2(X \otimes Y) \ar[d, "m_{X\otimes Y}"] \\
    T(X \otimes TY)  \ar[r, "T\mathtt{st}_{X,Y}"] & T^2(X \otimes Y) \ar[r, "m_{X\otimes Y}"] & T^2(X \otimes Y).
  \end{tikzcd}
  \]
  In this case we define a natural transformation, called \textit{double strenght}, with components
  $$ \mathtt{dst}_{X,Y} : TX\otimes TY \xrightarrow{\mathtt{st}_{TX,Y}} T(TX \otimes Y) \xrightarrow{T\mathtt{st}'_{X,Y}} T^2(X \otimes Y) \xrightarrow{m_{X\otimes Y}} T(X \otimes Y),$$
  or equivalently,
  $$\mathtt{dst}'_{X,Y}:  TX\otimes TY \xrightarrow{\mathtt{st}'_{X,TY}}   T(X \otimes TY)  \xrightarrow{T\mathtt{st}_{X,Y}} T^2(X \otimes Y) \xrightarrow{m_{X\otimes Y}} T(X \otimes Y).$$
\end{defin}
\begin{re}
  Notice that the map
  $$ t : \Pow{2}(X) \times \Pow{2}(Y) \rightarrow \Pow{2}(X  \times Y ), \mbox{ } (A,B) \mapsto A \times B,$$
  we defined in the previous section, is the double strenght of the following:
  $$ \mathtt{st}_{X,Y} : X \times \Pow{2}(Y) \rightarrow \Pow{2}(X \times Y), \mbox{ } (x, B) \mapsto \{x\} \times B.$$
\end{re}
\begin{defin}\label{BimDef}
Let $(C, \otimes, 1)$ be a symmetric monoidal category and let $(T,e,m)$  be a strong monad defined on it. Suppose $(X,\alpha), (Y, \beta), (Z,\gamma)$ are $T$-algebras. An arrow (in $C$) $f: X \otimes Y \rightarrow Z$ is called a \textit{bimorphism} if the following diagrams commute
\[
\begin{tikzcd}
  X \otimes TY \ar[d, "\Id \otimes \beta", swap] \ar[r, "\mathtt{st}_{X,Y}"] & T(X \otimes Y) \ar[r, "Tf"] & TZ \ar[d, "\gamma"] \\
  X \otimes Y \ar[rr, "f"] & & Z
\end{tikzcd}
\qquad
\begin{tikzcd}
  TX \otimes Y \ar[d, "\alpha \otimes \Id", swap] \ar[r, "\mathtt{st}'_{X,Y}"] & T(X \otimes Y) \ar[r, "Tf"] & TZ \ar[d, "\gamma"] \\
  X \otimes Y \ar[rr, "f"] & & Z.
\end{tikzcd}
\]
In this way, for all $X,Y \in C^T$, we define a functor
$$\Bim(X\otimes Y, = ) : C^T  \rightarrow \Sets,$$
where $\Bim(X \otimes Y, Z )$ denotes the sets of bimorphisms from $X\otimes Y$ to $Z$.
\end{defin}
The main result about strong monads we are interested in is contained in the following theorem.
\begin{teorema}\label{SCMon}\cite[Lemmas~5.1-5.3]{JACOBS199473}
Let $(T,e,m)$ be a strong monad on a symmetric monoidal category $(C,\otimes, 1)$ such that its associated category of algebras $C^T$ has coequalizers of reflexive pairs. Then, for each algebras $(X,\alpha)$, $(Y,\beta)$, $\Bim(X\otimes Y, = )$ is representable by an algebra $(X\otimes_T Y, \alpha \otimes_T \beta)$.\par\medskip
If additionally $(T,e,m)$ is commutative, then $C^T$ becomes a symmetric monoidal category with $\otimes_T$ as tensor product and with the free algebra $(T1 , m_1 )$ as the unit; moreover, the free functor $F : C \rightarrow C^T$ becomes strong monoidal. If $C$ has equalizers and its monoidal structure is closed, then also $C^T$ becomes a closed monoidal category.
\end{teorema}
\begin{re}
Let $(X,\alpha)$ and $(Y,\beta)$ be $T$-algebras. Their tensor product is the coequalizer of the following parallel pair of arrows
$$ T(T(X) \otimes T(Y)) \xrightarrow{T\mathtt{dst}_{X,Y}} T^2(X  \otimes Y )\xrightarrow{m} T(X  \otimes Y),$$
$$T(T(X) \otimes T(Y))\xrightarrow{T(\alpha \otimes \beta)} T(X  \otimes Y).$$
Note that this is exactly how we defined the tensor product of complete lattices.
\end{re}
\begin{re}
	Let $(T,m,e)$ be a monad with $T : \Sets \rightarrow \Sets$. Then, if we assume the axiom of choice, $\Sets^T$ is cocomplete (see \cite{Lin69, ABB+69}). Thus, Eilenberg-Moore categories for strong monads defined on $\Sets$ always satisfy the hypotesis of Theorem \ref{SCMon}.
\end{re}
\begin{re}\label{Unitor}
We obtain the associator in $C^T$ from the one in $C$ by using the universal property of bimorphisms. Similarly, we can obtain the unitors in $C^T$ by using the tensorial strength. As an example, the left unitor at an object $(X,\alpha)$ is the arrow associated to the bimorphism
$$ T1 \otimes X \xrightarrow{st'_{1,X}} T(1\otimes X) \xrightarrow{\simeq} T(X) \xrightarrow{\alpha} X.$$
\end{re}
\section{Applications}
The $V$-powerset monad $(\Pow{V}, u, n)$ is the enriched generalization of the classical powerset monad, where we define $\Pow{V} : \Sets \rightarrow \Sets$ by putting $\Pow{V}(X) =  V^X$ and, for $f: X \rightarrow Y$ and $\phi \in V^X$
$$ \Pow{V}(f)(\phi)(y) =  \bigvee_{x \in f^{-1}(y) }\phi(x).$$
Moreover:
\begin{itemize}
	\item $u_X : X \rightarrow V^X$ is the transpose of the diagonal $\bigtriangleup_X : X \times X \rightarrow V$;
	\item $n_X : \Pow{V}(\Pow{V}(X)) \rightarrow \Pow{V}(X)$ is defined by $n_X(\Phi)(x) =  \bigvee_{\phi \in V^X} \Phi(\phi) \otimes \phi(x)$.
\end{itemize}
Consider $\Sets$ as a monoidal category in the usual way, that is to say, with its cartesian structure and consider the following function
$$\mathtt{st}_{X,Y} : X \times \Pow{V}Y \rightarrow  \Pow{V}(X \times Y), \mbox{ } (x,\phi) \mapsto (u_X(x)\boxtimes \phi),$$
where $(u_X(x)\boxtimes \phi)(\tilde{x},y) =  u_X(x)(\tilde{x}) \otimes \phi(y).$\par \medskip
Long and boring computations that someone must do, show that this makes $(\Pow{V}, u, n)$ into a strong monad. Indeed, the commutativity of
\[
\begin{tikzcd}
&  \Pow{V}Y &  \\
1 \times \Pow{V}Y \ar[ur, "\simeq"] \ar[rr, "\mathtt{st}_{1,Y}"]&  & \Pow{V}(X \times Y) \ar[ul, "\simeq", swap]
\end{tikzcd}
\]
is straightforward. While, from $u_{X\times Y} \simeq u_X \boxtimes u_Y$, it follows that the diagram
\[
\begin{tikzcd}[row sep=large, column sep=large]
 (X \times Y) \times Z \ar[rr, "\mathtt{st}_{X \times Y,Z}"] \ar[d, "\simeq"] &  & \Pow{V}((X \times Y) \times Z) \ar[d, "\simeq"] \\
 X \times (Y \times \Pow{V}Z)\ar[r, "\Id \times \mathtt{st}_{Y,Z}"] & X \times \Pow{V}(Y \times Z) \ar[r, "\mathtt{st}_{X,Y \times Z}"] & \Pow{V}(X \times (Y \times Z))
\end{tikzcd}
\]
commutes, since
\[
\begin{tikzcd}
((x,y),\psi) \ar[rr, mapsto] \ar[d,  mapsto] &  &  u_{X \times Y} \boxtimes \psi \ar[d,  mapsto] \\
  (x,(y,\psi)) \ar[r,  mapsto] & (x,(u_Y(y) \boxtimes \psi) \ar[r,  mapsto] & u_X(x) \boxtimes u_Y(y) \boxtimes \psi.
\end{tikzcd}
\]
In the same way it follows that the diagram
\[
\begin{tikzcd}
&  X \times Y \ar[dr, "u_{X \times Y}"]  \ar[dl, "\Id \times u_Y", swap] &  \\
X \times \Pow{V}Y  \ar[rr, "\mathtt{st}_{X,Y}"]&  & \Pow{V}(X \times Y)
\end{tikzcd}
\]
commutes, since
\[
\begin{tikzcd}
&  (x,y) \ar[dr, mapsto]  \ar[dl, mapsto] &  \\
(x,u_Y(y))  \ar[rr, mapsto]&  & u_X \boxtimes u_Y.
\end{tikzcd}
\]
To verify the commutativity of
\[
\begin{tikzcd}[row sep=large, column sep=large]
  X \times \Pow{V}^2Y \ar[r, "\mathtt{st}_{X,\Pow{V}Y}"] \ar[d, "\Id \times n_Y", swap] & \Pow{V}(X \times TY )  \ar[r, "\Pow{V}\mathtt{st}_{X,Y}"] & \Pow{V}^2(X \times Y) \ar[d, "n_{X \times Y}"], \\
  X \times \Pow{V}Y \ar[rr, "\mathtt{st}_{X,Y}"] &  & \Pow{V}(X \times Y)
\end{tikzcd}
\]
requires a little bit of effort. We have to show that
\[
\begin{tikzcd}
  (x, \Psi) \ar[r, mapsto] \ar[d, mapsto] &  u_X(x) \boxtimes \Psi \ar[r, mapsto] & \Pow{V}\mathtt{st}_{X,Y}(u_X(x) \boxtimes \Psi)  \ar[d, mapsto], \\
  (x, n_Y(\Psi)) \ar[rr, mapsto] &  &  u_X(x) \boxtimes n_Y(\Psi) = n_{X\times Y}(\Pow{V}\mathtt{st}_{X,Y}(u_X(x) \boxtimes \Psi) ).
\end{tikzcd}
\]
Here $ u_X(x) \boxtimes n_Y(\Psi) = n_{X\times Y}(\Pow{V}\mathtt{st}_{X,Y}(u_X(x) \boxtimes \Psi) )$ follows by us unravelling the definitions of $\Pow{V}\mathtt{st}_{X,Y}$ and $n$, and by noticing that $\mathtt{st}_{X,Y} \cdot \mathtt{ev}_{X \times Y} = \Delta_X \boxtimes  \mathtt{ev}_Y$, as we have
\begin{alignat*}{2}
  \mathtt{st}_{X,Y} \cdot \mathtt{ev}_{X \times Y}((\tilde{x},\psi), (x,y)) & = \mathtt{ev}_{X \times Y}(u_X(\tilde{x}) \boxtimes \psi,(x,y) ) \\
  & = u_X(\tilde{x})(x) \boxtimes \psi(y) \\
  & = \Delta_X (\tilde{x}, x)\boxtimes  \mathtt{ev}_Y(\psi,y) \\
  & = \Delta_X \boxtimes  \mathtt{ev}_Y((\tilde{x},\psi), (x,y)).
\end{alignat*}
\begin{re}
  Notice that the strongness of $(\Pow{V}, u, n)$ follows from the fact that every functor $F : \Sets \rightarrow \Sets$ is a $\Sets$-functor (where the monoidal structure on $\Sets$ is the usual one), and from the fact that to give a $\Sets$-enrichment, for a $\Sets$-monad $(T,e,m)$, is equivalent to give a strenght (see Propositions $1.1,1.2$ of \cite{AKSM}).\\
  We preferred to give an explicit treatment since the calculations involved, despite being boring, were not too long and complicated.
\end{re}

Moreover, since we always assume our base quantale $V$ to be commutative, it is easy to show that $(\Pow{V}, u, n)$ is also commutative, with the co-strenght $\mathtt{st}'$ given by
$$\mathtt{st'}_{X,Y} : \Pow{V}X\times Y \rightarrow \Pow{V}(X \times Y), \mbox{ } (\psi, y) \mapsto \psi \boxtimes u_Y.$$
Indeed we have that the diagram
\[
\begin{tikzcd}
  \Pow{V}X\times \Pow{V}Y \ar[r, "\mathtt{st}_{\Pow{V}X,Y}"] \ar[d,"\mathtt{st}'_{X,\Pow{V}Y}", swap] & \Pow{V}(\Pow{V}X \times Y) \ar[r, "\Pow{V}\mathtt{st}'_{X,Y}"] &  \Pow{V}^2(X \times Y) \ar[d, "n_{X \times Y}"]\\
  \Pow{V}(X \times \Pow{V}Y)  \ar[r, "\Pow{V}\mathtt{st}_{X,Y}"] & \Pow{V}^2(X \times Y) \ar[r, "n_{X \times Y}"] & \Pow{V}(X \times Y)
\end{tikzcd}
\]
commutes, since
\[
\begin{tikzcd}
   (\psi, \phi) \ar[r, mapsto]  \ar[d, mapsto] &  u_{\Pow{V}X}(\psi) \boxtimes \psi  \ar[r, mapsto] &  \Pow{V}\mathtt{st}'_{X,Y}(u_{\Pow{V}X}(\psi) \boxtimes \psi) \ar[d, mapsto] \\
    \psi \boxtimes u_{\Pow{V}Y}(\phi) \ar[r, mapsto] &  \Pow{V}\mathtt{st}_{X,Y}(  \psi \boxtimes u_{\Pow{V}Y}(\phi))  \ar[r, mapsto] & \psi \boxtimes \phi.
\end{tikzcd}
\]
Here $$n_{X\times Y}(\Pow{V}\mathtt{st}_{X,Y}(  \psi \boxtimes u_{\Pow{V}Y}(\phi))) = n_{X \times Y}(\Pow{V}\mathtt{st}'_{X,Y}(u_{\Pow{V}X}(\psi) \boxtimes \psi)) = \psi \boxtimes \phi$$ follows from $\mathtt{st}_{X,Y} \cdot \mathtt{ev}_{X \times Y} = \Delta_X \boxtimes  \mathtt{ev}_Y$, from $\mathtt{st}'_{X,Y} \cdot \mathtt{ev}_{X \times Y} = \mathtt{ev}_X \boxtimes \Delta_Y $, and from the monad law $n \cdot u_{\Pow{V}(-)} = \Id_{(-)}$.\\

By applying Theorem \ref{SCMon} we get the following result.
\begin{prop}\label{StrongPow}
  The category $\Alg{V}$ of algebras for the $V$-powerset monad $(\Pow{V}, u, n)$ admits a symmetric closed monoidal structure $\Tensor{\Pow{V}}$ with unit given by $\Pow{V}(1) = V$ such that the free functor
  $$\Pow{V} : \Sets \rightarrow \Alg{V}, \mbox{ } X \mapsto (\Pow{V}(X), n_X)$$
  becomes strong monoidal. Moreover, $\Tensor{\Pow{V}}$ classifies bimorphisms in the sense of Definition \ref{BimDef}.
\end{prop}

The last thing we have to do is to tune a little bit more the notion of bimorphism in our particular case, in order to have a more manageable formulation. As the toy example we played with in the first section suggests, the notion of bimorphism in categories in which the notion of "point" resembles the one in $\Sets$ seems to reduce to the "componentwise preserving structure" notion like the one we have in $\mathtt{Sup}$ and the one in algebra. This motivates us to introduce the following definition.
\begin{defin} \label{BimDef2}
Suppose $(X,\alpha), (Y, \beta), (Z,\gamma)$ are in $\Alg{V}$. A function $f: X \times Y \rightarrow Z$ is called a bimorphism if the following diagrams commute, for all $x,y \in X,Y$,
\[
\begin{tikzcd}[row sep=large, column sep=large]
\Pow{V}(X) \ar[r, "\simeq"] \ar[d, "\alpha",swap] & \Pow{V}(X )\times 1 \ar[d, "\alpha \times 1",swap] \ar[r, "\Pow{V}(\Id \times y)"] & \Pow{V}(X\times Y )  \ar[r, "\Pow{V}f"] & \Pow{V}(Z) \ar[d, "\gamma"]\\
X  \ar[r, "\simeq"] & X \times 1 \ar[r, "\Id \times y"] & X \times Y \ar[r, "f"] & Z
\end{tikzcd}
\]
\[
\begin{tikzcd}[row sep=large, column sep=large]
\Pow{V}(Y) \ar[r, "\simeq"] \ar[d, "\beta", swap] & 1 \times  \Pow{V}(Y ) \ar[d, " 1 \times \beta",swap] \ar[r, "\Pow{V}(x \times \Id )"] & \Pow{V}(X\times Y )  \ar[r, "\Pow{V}f"] & \Pow{V}(Z) \ar[d, "\gamma"]\\
Y  \ar[r, "\simeq"] & 1 \times Y \ar[r, "x \times \Id"] & X \times Y \ar[r, "f"] & Z.
\end{tikzcd}
\]
\end{defin}
\begin{re}
  Notice that, for $V = \mathbf{2}$, we recover the notion of bimorphism we gave for complete lattices in the first section.
\end{re}
Now we have not only one but two notions of bimorphism! Of course, as one might expect, the two notions coincide.
\begin{prop}\label{EquiBimo}
Suppose $(X,\alpha), (Y, \beta), (Z,\gamma)$ are in $\Alg{V}$. A function $f: X \times Y \rightarrow Z$ is a bimorphism according to Definition \ref{BimDef} iff it is so according to Definition \ref{BimDef2}.
\end{prop}
\begin{proof}
  Let us do the case in which we "fix" $y \in Y$, the other one is similar.\\
  The proof follows by contemplating the following diagram
\[
\begin{tikzcd}
         & \Pow{V}(X\times 1) \ar[rrrdd, bend left, "\Pow{V}(1_X \times y)"] &                         &  &  &   \\
         &  \clubsuit             &                        &  \diamondsuit &    & \\
\Pow{V}(X) \ar[uur, "\simeq"] \ar[dd, "\alpha",swap] \ar[rr, "\simeq"] &                 &   \Pow{V}(X) \times 1 \ar[uul, "\mathtt{st}'", swap]  \ar[r,"1_{\Pow{V}X} \times y"] \ar[dd, "\alpha \times 1", swap] &   \Pow{V}(X) \times Y \ar[dd, "\alpha \times \Id_Y",swap] \ar[r,"\mathtt{st}'"] & \Pow{V}(X \times Y) \ar[r,"\Pow{V}f"] & \Pow{V}(Z)  \ar[dd, "\gamma"] \\
          &                         &  &  & \spadesuit \\
    X  \ar[rr, "\simeq"] &  & X \times 1 \ar[r, "1 \times y"] & X \times \ar[rr, "f"]Y &  & Z.
\end{tikzcd}
\]
Here $\clubsuit$ commutes since $\Pow{V}$ is a strong monad while $\diamondsuit$ commutes because $\mathtt{st}'$ is a natural transformation.\par\medskip
Suppose $f$ is a bimorphism according to \ref{BimDef}, then $\spadesuit$ commutes, hence the outer diagram too. This implies that $f$ is a bimorphism according to \ref{BimDef2} too.\par\medskip
If $f$ is a bimorphism according to \ref{BimDef2}, then the outer diagram commutes, hence, for all $y\in Y$, we have
$$\gamma \cdot \Pow{V}f \cdot \mathtt{st}'\cdot 1_{\Pow{V}X} \times y = f \cdot \alpha \times \Id_Y \cdot 1_{\Pow{V}X} \times y.$$
Since $( 1_{\Pow{V}X} \times y : \Pow{V}(X) \times 1 \rightarrow \Pow{V}(X) \times Y)_{y \in Y},$ is a jointly epic family, we can jointly cancel them in the previous equation. Thus we obtain the commutativity of $\spadesuit $ which implies that $f$ is a bimorphism according to \ref{BimDef} too.\\
\end{proof}
\begin{re}
In Theorem \ref{SetMon} we proved that the category of algebras for this monads is equivalent to $\Coco{\Vcats}$. Hence the monoidal structure on $\Alg{V}$ transfers to a monoidal structure on $\Coco{\Vcats}$.\par\medskip
In particular, from the previous proposition, and since the equivalence $$\Alg{V} \simeq \Coco{\Vcats}$$ changes only the corresponding structures (and it leaves the underlying sets and arrows unchanged), we have that a $V$-functor $f : (X,a) \boxtimes (Y,b) \rightarrow (Z,c)$ is a bimorphism if, for all $x, y \in X,Y$, one has
$$f_x : (Y,b) \rightarrow (Z,c), \mbox{ } y \mapsto f(x,y),$$
$$ f_y : (X,a) \rightarrow (Z,c), \mbox{ } x \mapsto f(x,y),$$
are cocontinuous $V$-functor.\par\medskip
Since these kind of bimorphisms are classified by a monoidal structure on $\Coco{\Vcats}$, as described at the end of \cite{ECT}, by arguments similar to the one we used in Proposition \ref{MonBim}, we get that the monoidal structure on $\Coco{\Vcats}$, induced by the equivalence $\Alg{V} \simeq \Coco{\Vcats}$, and the one studied in \cite{ECT} coincide.
\end{re}


\end{document}